\documentclass[12pt,a4paper]{article}

\usepackage[utf8]{inputenc}
\usepackage[T1]{fontenc}
\usepackage[a4paper,margin=1in]{geometry}
\usepackage{amsmath,amssymb,amsthm,amscd,mathrsfs,mathtools}
\usepackage{tikz-cd}
\usetikzlibrary{calc,babel}
\usepackage{stmaryrd}
\usepackage{url}
\usepackage[colorlinks=true,linkcolor=blue,citecolor=blue,urlcolor=blue]{hyperref}
\usepackage{faktor}

\DeclareMathOperator{\sgn}{sgn}
\DeclareMathOperator{\Tr}{Tr}

\DeclareMathOperator{\Hom}{Hom}

\DeclareMathOperator{\Gal}{Gal}

\DeclareMathOperator{\GL}{GL}
\DeclareMathOperator{\SL}{SL}

\DeclareMathOperator{\Frac}{Frac}
\DeclareMathOperator{\Sym}{Sym}
\DeclareMathOperator{\aug}{aug}
\DeclareMathOperator{\res}{res}
\DeclareMathOperator{\vol}{vol}
\DeclareMathOperator{\ev}{ev}
\DeclareMathOperator{\Eis}{Eis}
\DeclareMathOperator{\pol}{pol}
\DeclareMathOperator{\id}{id}

\DeclareMathOperator{\supp}{supp}
\DeclareMathOperator{\Avg}{Avg}
\DeclareMathOperator{\diag}{diag}

\numberwithin{equation}{section}

\theoremstyle{plain}
\newtheorem{theorem}[equation]{Theorem}
\newtheorem{lemma}[equation]{Lemma}
\newtheorem{corollary}[equation]{Corollary}
\newtheorem{proposition}[equation]{Proposition}

\theoremstyle{definition}
\newtheorem{definition}[equation]{Definition}

\theoremstyle{remark}
\newtheorem{remark}[equation]{Remark}

\title{Shintani Cocycle Decomposition of Topological Polylogarithms for Totally Real Fields}
\author{Julio de Mello Bezerra}
\date{October 21, 2026}

\begin{document}

\maketitle

\begin{abstract}
We compare two cohomological constructions computing special values of partial zeta functions of totally real number fields at nonpositive integers: the Shintani cocycles arising from polyhedral cone decompositions, and the Eisenstein classes obtained from the topological polylogarithm. In particular we prove that the specialization of the topological polylogarithm at a nonzero torsion point coincides in cohomology with the Shintani class up to an explicit sign and error term, and obtain Siegel and Klingen's classical rationality Theorem.
\end{abstract}

\tableofcontents

\section{Introduction}\label{sec:intro}

\subsection{Background and context}

The study of special values of zeta and $L$-functions at nonpositive integers is a central theme in algebraic number theory, originating in Euler's discovery that the Riemann zeta function takes rational values $\zeta(-k) = (-1)^k \frac{B_{k+1}}{k+1} \in \mathbb{Q}$ ($k \geq 0$) given by the Bernoulli numbers. For a totally real number field $F$ of degree $n = [F:\mathbb{Q}]$, Siegel and Klingen \cite{Sie70} proved that the partial zeta functions $\zeta(\mathfrak{b}, \mathfrak{f}, s)$ associated to narrow ray classes of ideals take rational values at all nonpositive integers $s = -k$.

Over the past several decades, two distinct geometric and cohomological paradigms have been developed to study these special values and prove their rationality:

\begin{enumerate}
    \item \textit{Shintani's method and cone cocycles.} In \cite{Shi76}, Shintani proved that the action of the group of totally positive units $\Gamma \subset \mathcal{O}_F^{\times,+}$ on the totally positive orthant $V_{>0} = (F \otimes_{\mathbb{Q}} \mathbb{R})_{>0} \cong \mathbb{R}_{>0}^n$ admits a fundamental domain given by a finite union of rational polyhedral cones $\sigma_{\boldsymbol{\alpha}}$. Decomposing the lattice sum defining $\zeta(\mathfrak{b}, \mathfrak{f}, s)$ across these cones expresses the partial zeta function as a finite sum of cone zeta functions, whose generating series $\mathcal{G}_{\sigma_{\boldsymbol{\alpha}}}(\mathbf{t}, \mathbf{z})$ are explicit rational functions of exponentials $e^{\langle \alpha, \mathbf{z} \rangle}$. Although the choice of a Shintani cone decomposition $\Phi$ is non-canonical, the resulting cone generating series assemble into a canonical cohomology class $[\mathcal{G}]$ in the $(n-1)$-st group cohomology of $\Gamma$ (studied in various forms, in particular by Charollois--Dasgupta--Greenberg \cite{CDG15}, and Bannai--Hagihara--Yamada--Yamamoto \cite{BHY19}).
    \item \textit{Eisenstein cohomology and the topological polylogarithm.} Independently, Sczech \cite{Scz93} and Nori \cite{Nor95} constructed Eisenstein cohomology classes for $\GL_n(\mathbb{Z})$ using analytic trigonometric series and differential forms. Reinterpreting Nori's construction in the framework of sheaf theory, Beilinson, Kings, and Levin \cite{BKL18} (generalizing \cite{BL94}) introduced the \textit{logarithm sheaf} $\mathcal{L}\mathrm{og}$ on the real torus $T = V/L$ with coefficients in the completed group ring $R_A = \varprojlim_k A[L]/J^k$ for $A$ a Noetherian ring. In their approach, the \textit{topological polylogarithm} is characterized purely topologically by a residue condition on the punctured torus $\overset{\circ}{T} := T \setminus \{0\}$, and its specialization along a nonzero torsion section $\mathbf{t} \in \overset{\circ}{T}$ produces the Eisenstein class $\Eis(\mathbf{t})$ in the cohomology of $\Gamma$, avoiding any choice of cone decomposition.
\end{enumerate}

A natural question is to understand the precise relationship between the topological polylogarithm of Beilinson--Kings--Levin \cite{BKL18} and the Shintani generating class $[\mathcal{G}]$ at the cohomological level. It must also be noted that both approaches above also obtain integrality results and $p$-adic interpolation of the special zeta values. Our simplified framework in this article does not yield such integrality results, though modifications should be possible and may be considered in subsequent work. 

\subsection{Main results}

In this paper, we establish a direct cohomological and analytic comparison between the topological polylogarithm and the Shintani cocycle class, by changing the framework of \cite{BKL18} in a way which is compatible with Shintani decompositions. Our approach rests on three main ingredients:

\paragraph{1. Cone-adapted Laurent series and the extended logarithm sheaf.}
Because the individual Shintani cone generating functions $\mathcal{G}_{\sigma_{\boldsymbol{\alpha}}}(\mathbf{t}, \mathbf{z})$ have poles along the hyperplanes $\langle \alpha_j, \mathbf{z} \rangle = 0$ orthogonal to the generators of the cone $\sigma_{\boldsymbol{\alpha}}$, they do not lie in the power series ring $R_{\mathbb{K}} \cong \mathbb{K}\llbracket \mathbf{z} \rrbracket$ of \cite{BKL18}. To circumvent this issue, as well as the pathologies arising from working with the usual ring of multivariate Laurent series, we embed $R_{\mathbb{K}}$ into the field $\widetilde{R}_{\mathbb{K}} := \mathbb{K}_{\leq_{\mathrm{lex}}}(\!(\mathbf{z})\!)$ of formal Laurent series supported in cones compatible with the lexicographic order, which carries a natural diagonal action of $\Gamma$ and a norm projection $\rho_N: \widetilde{R}_{\mathbb{K}} \twoheadrightarrow (\widetilde{R}_{\mathbb{K}})_\Gamma \cong \mathbb{K}(\!(w)\!)$ onto $\Gamma$-coinvariants ($w = z_1 \cdots z_n$).

Extending scalars along $R_{\mathbb{K}} \hookrightarrow \widetilde{R}_{\mathbb{K}}$ defines the \textit{extended logarithm sheaf} $\widetilde{\mathcal{L}\mathrm{og}} := \mathcal{L}\mathrm{og} \otimes_{\underline{R}} \underline{\widetilde{R}}$ on $T$. Over this extended sheaf, the topological residue map on the punctured torus $\overset{\circ}{T} = T \setminus \{0\}$ is an isomorphism:
\[
    \res: H^{n-1}(\overset{\circ}{T}, \widetilde{\mathcal{L}\mathrm{og}}) \xrightarrow{\sim} \widetilde{\mathcal{L}\mathrm{og}}_0 = \widetilde{R}_{\mathbb{K}}.
\]
This allows us to define the topological polylogarithm class $\pol \in H^{n-1}(\overset{\circ}{T}, \Gamma; \widetilde{\mathcal{L}\mathrm{og}})$ directly by the residue condition $\res(\pol) = \delta_0 = 1$ (Definition \ref{def polylog}), without twisting by a degree-zero divisor $\alpha$ or by $L^*$. We must note that while this simplifies the theory, we do not obtain integrality and $p$-adic interpolation results as in \cite{BKL18} and \cite{CDG15}.

\paragraph{2. Explicit current representing the topological polylogarithm.}
To represent $\pol$ explicitly in de Rham cohomology, we first prove a general comparison theorem (Theorem \ref{residue currents} and Corollary \ref{cor:explicit_residue}) identifying the relative de Rham cocone $C(j^*\Omega_X^\bullet)$ of an open complement $j: U = X \setminus Z \hookrightarrow X$ with the complex $\Gamma_Z(X, \widehat{\Omega}_X^\bullet)$ of currents supported on $Z$ and establishing the explicit residue formula $\res([\theta]) = [d\mathcal{T}_\theta - \mathcal{T}_{d\theta}]$ for locally integrable forms.

Applying the above comparison to the totally positive unit sphere $X := \mathbb{S}^{n-1} \cap (\mathbb{R}_{>0})^n$, we follow a method due to Prader to explicitly construct the current $\widetilde{\mathscr{G}}$ in Proposition \ref{prop green current}, which satisfies the desired residue condition. In particular, $\widetilde{\mathscr{G}}$ directly represents the topological polylogarithm class $\pol$.

\paragraph{3. Simplicial evaluation and comparison with the Shintani class.}
Any Shintani cone decomposition $\Phi$ induces a $\Gamma$-equivariant triangulation $\Phi^X$ of $X$ into spherical simplices $X_{\boldsymbol{\alpha}} = \sigma_{\boldsymbol{\alpha}} \cap X$. After reinterpreting Hurwitz's classical integral formula (Proposition \ref{hurwitz}) as a Mellin transform, we integrate the specialization $\mathbf{t}^*\widetilde{\mathscr{G}}$ over each simplex $X_{\boldsymbol{\alpha}}$ and apply Poisson summation on each cone (Proposition \ref{prop: equality of distributions}) to obtain our main result:

\begin{theorem}[Theorem \ref{main theorem}]
Let $\mathbf{t} \in \overset{\circ}{T}$ be a nonzero torsion point of the punctured torus stabilized by $\Gamma$. Let $\Eis(\mathbf{t}) := \mathbf{t}^*\pol = [\mathbf{t}^*\widetilde{\mathscr{G}}] \in H^{n-1}(\Gamma, \widetilde{R}_\Gamma)$ be the Eisenstein class obtained by pulling back the topological polylogarithm $\pol$ along $\mathbf{t}$, and let $[\mathcal{G}] \in H^{n-1}(\Gamma, \widetilde{R}_\Gamma)$ be the Shintani class. Then in $H^{n-1}(\Gamma, \widetilde{R}_\Gamma)$ we have
\[
    \Eis(\mathbf{t}) = (-1)^{n-1}[\mathcal{G}].
\]
Pairing with the fundamental class, the evaluation map $\ev: H^{n-1}(\Gamma, \widetilde{R}) \to \widetilde{R}_\Gamma$ yields
\[
    \ev(\Eis(\mathbf{t}) + [\Psi_P]) = \ev\bigl((-1)^{n-1}[\mathcal{G}] + [\Psi_P]\bigr) = (-1)^{n-1} \sum_{k \geq 0} \zeta(\mathbf{t}, L, -k) \frac{w^k}{(k!)^n} \in \mathbb{Q}\llbracket w \rrbracket,
\]
where $[\Psi_P]$ is the explicit polar cocycle of \cite{CDG15} and \cite{Scz93}.  
\end{theorem}

The Klingen--Siegel rationality theorem (Corollary \ref{cor: klingen siegel}) follows immediately.

\subsection{Outline of the paper}

In Section \ref{sec:motivation}, we motivate our constructions from the analytic point of view by reviewing how special values of the Riemann, Hurwitz, and Shintani zeta functions at nonpositive integers are extracted via Mellin transforms of their generating functions.

In Section \ref{sec:polylog}, we set up the geometric framework for a totally real field $F$ of degree $n$, introduce the logarithm sheaf $\mathcal{L}\mathrm{og}$ on $T = V/L$ and the positive orthant sphere $X = \mathbb{S}^{n-1} \cap V_{>0}$, and prove Theorem \ref{residue currents} and Corollary \ref{cor:explicit_residue} on relative de Rham cohomology and residues of currents. We then introduce the cone-adapted Laurent series field $\widetilde{R}_{\mathbb{K}} = \mathbb{K}_{\leq_{\mathrm{lex}}}(\!(\mathbf{z})\!)$, define the extended logarithm sheaf $\widetilde{\mathcal{L}\mathrm{og}}$, and construct the topological polylogarithm $\pol$ (Definition \ref{def polylog}).

In Section \ref{sec:shintani}, we review Shintani cone decompositions $\Phi$ of $V_{>0} \cup \{0\}$, express cone zeta functions in terms of the generating series $\mathcal{G}_{\sigma_{\boldsymbol{\alpha}}}(\mathbf{t}, \mathbf{z})$, and use the associated simplicial resolution $C_\bullet(\Phi)$ to define the Shintani cohomology class $[\mathcal{G}]$ (Definition \ref{def shintani class}).

In Section \ref{sec:explicit_comparison}, we prove Hurwitz's integral formula on $X$ as a Mellin transform on $V_{>0}$ (Proposition \ref{hurwitz}), establish the Poisson summation identity relating Sczech's rational functions to $\mathcal{G}_{\sigma_{\boldsymbol{\alpha}}}(\mathbf{t}, \mathbf{z})$ (Proposition \ref{prop: equality of distributions}), and give the direct construction of the Green currents $\mathscr{G}$ and $\widetilde{\mathscr{G}}$ on $X \times T$ via fibre integration of the theta-form $\theta_\mu$ (Lemma \ref{lem:fibre_integral_green} and Proposition \ref{prop green current}). Evaluating $\mathbf{t}^*\mathscr{G}$ and $\mathbf{t}^*\widetilde{\mathscr{G}}$ on the simplices of $\Phi^X$ gives our main Theorem \ref{main theorem} and Corollary \ref{cor: klingen siegel}.

Finally, in Section \ref{sec:literature_comparison}, we provide a detailed comparison of our framework and results with the polylogarithm and Green current constructions of Beilinson, Kings, and Levin \cite{BKL18} (Section \ref{subsection:BKL comparison}) and with the Shintani cocycle, Solomon--Hu pairing, and Shintani operator of Charollois, Dasgupta, and Greenberg \cite{CDG15} (Section \ref{subsection:CDG comparison}).

\subsection{Acknowledgements}
This paper is based on my doctoral thesis and would not have been possible without the guidance and support of my advisor Guido Kings. I also acknowledge the help and direct influence of Johannes Sprang, Han-Ung Kufner and Lukas Prader.

\section{Motivation: Zeta Functions at Nonpositive Integers}\label{sec:motivation}

\paragraph{Riemann zeta function.}
We begin by briefly recalling the relationship between the \textit{Riemann zeta function} $\zeta$ and the \textit{Gamma function} $\Gamma$, defined respectively by
\[
    \zeta(s) := \sum_{n=1}^{\infty} n^{-s} \quad \text{and} \quad \Gamma(s) := \int_{0}^{\infty} e^{-u} u^{s-1}\,du,
\]
for $s \in \mathbb{C}$ with $\Re(s)>1$ and $\Re(s)>0$, respectively. Both functions admit meromorphic continuations to $\mathbb{C}$ (see \cite[Th\'eor\`eme VII.2.1]{Col09} and Lemma \ref{Lemma: mellin} below).

For $\lambda \in \mathbb{R}_{>0}$, the change of variables $u = \lambda t$ yields
\begin{equation}\label{eq: gamma}
    \lambda^{-s}\Gamma(s) = \int_{0}^{\infty} e^{-t\lambda} t^{s-1}\,dt,
\end{equation}
which extends by contour integration to $\lambda \in \mathbb{C}$ with $\Re(\lambda) > 0$. Setting $\lambda = n$, summing over $n \ge 1$, and interchanging summation and integration gives
\[
    \zeta(s)\Gamma(s) = \int_{0}^{\infty} \sum_{n=1}^{\infty} e^{-nt} t^{s-1}\,dt.
\]
Thus, setting $G: \mathbb{R}_{>0} \to \mathbb{R}$, $G(t) := \frac{e^{-t}}{1-e^{-t}} = \sum_{n=1}^{\infty} e^{-nt}$, we obtain the \textit{Mellin transform} representation
\begin{equation}\label{eq: mellin}
    \zeta(s)\Gamma(s) = \int_{0}^{\infty} G(t) t^{s-1}\,dt.
\end{equation}

The following classical lemma provides the analytic continuation and special values at nonpositive integers of Mellin transforms of rapidly decreasing smooth functions.

\begin{lemma}\label{Lemma: mellin}
Suppose $f \in C^{\infty}(\mathbb{R}_{\geq 0})$ and all of its derivatives decrease rapidly at infinity (i.e., $f^{(m)}(x) = O(|x|^{-N})$ as $x \to \infty$ for all $m, N \in \mathbb{N}$). Then
\[
    M(f,s) := \frac{1}{\Gamma(s)} \int_{0}^{\infty} f(t) t^{s-1}\,dt,
\]
defined initially for $\Re(s)>0$, admits an analytic continuation to an entire function on $\mathbb{C}$, and for $k \in \mathbb{N}$,
\[
    M(f,-k) = (-1)^k \left(\frac{d}{dt}\right)^k f(t) \Big|_{t=0}.
\]
\end{lemma}

\begin{proof}
We follow \cite[Proposition VII.2.6]{Col09}. The rapid decay of $f$ ensures that $\int_{0}^{\infty} |f(t) t^{s-1}|\,dt$ converges uniformly on compact subsets of $\{s \in \mathbb{C} \mid \Re(s)>0\}$, and since $\Gamma(s)$ is nonvanishing and holomorphic there, $M(f,s)$ is holomorphic for $\Re(s)>0$.

Integrating by parts gives
\[
    M(f',s+1) + M(f,s) = \frac{1}{\Gamma(s+1)} \int_{0}^{\infty} f'(t) t^{s}\,dt + \frac{1}{\Gamma(s+1)} \int_{0}^{\infty} f(t) (t^{s})'\,dt = \frac{1}{\Gamma(s+1)} \bigl[f(t)t^{s}\bigr]^{\infty}_0 = 0
\]
for $\Re(s)>0$. By induction, $M(f,s) = (-1)^k M(f^{(k)},s+k)$ for $\Re(s)>0$ and $k \in \mathbb{N}$, which provides the analytic continuation of $M(f,s)$ to $\Re(s) > -k$ for all $k \in \mathbb{N}$. Evaluating at $s = -k$ yields
\[
    M(f,-k) = (-1)^{k+1} M(f^{(k+1)},1) = (-1)^{k+1} \int_{0}^{\infty} f^{(k+1)}(t)\,dt = (-1)^{k+1} \bigl[f^{(k)}(t)\bigr]^\infty_0 = (-1)^k f^{(k)}(0). \qedhere
\]
\end{proof}

\begin{remark}
In the language of distributions, for each $s \in \mathbb{C}$, $f \mapsto M(f,s)$ defines a tempered distribution $\mathcal{M}_s$ on the Schwartz space $\mathcal{S}(\mathbb{R}_{\geq 0})$. For $k \in \mathbb{N}$, $\mathcal{M}_{-k} = \delta_0^{(k)}$ is the $k$-th derivative of the Dirac delta distribution at $0$.
\end{remark}

Since $G(t)$ has a simple pole at $t=0$, we apply Lemma \ref{Lemma: mellin} to $F(t) := tG(t)$ via
\[
    \zeta(s) = \frac{1}{\Gamma(s)} \int_{0}^{\infty} G(t) t^{s-1}\,dt = \frac{1}{(s-1)\Gamma(s-1)} \int_{0}^{\infty} F(t) t^{s-2}\,dt = \frac{1}{s-1} M(F,s-1).
\]
Lemma \ref{Lemma: mellin} then yields $\zeta(-k) = (-1)^k \frac{B_{k+1}}{k+1}$, where the \textit{Bernoulli numbers} $B_n \in \mathbb{Q}$ are defined by
\[
    \sum_{n=0}^{\infty} B_n \frac{t^n}{n!} = \frac{t}{e^t-1} = \frac{te^{-t}}{1-e^{-t}} = F(t).
\]

\paragraph{Hurwitz zeta function and Dirichlet $L$-series.}
For a multiplicative Dirichlet character $\chi: (\mathbb{Z}/N\mathbb{Z})^\times \to \mathbb{C}^\times$ (extended by zero to $\mathbb{Z}$) and $0 < x \leq 1$, the \textit{Dirichlet $L$-function} and \textit{Hurwitz zeta function} are defined for $\Re(s) > 1$ by
\[
    L(s,\chi) := \sum_{n=1}^{\infty} \chi(n) n^{-s}, \qquad \zeta(s,x) := \sum_{n=0}^{\infty} (n+x)^{-s},
\]
and are related by the decomposition
\[
    L(s,\chi) = \sum_{a=1}^N \chi(a) N^{-s} \zeta\left(s, \frac{a}{N}\right).
\]
Setting $G(t,x) := \frac{e^{-xt}}{1-e^{-t}} = \sum_{n=0}^{\infty} e^{-(n+x)t}$, we have $\Gamma(s)\zeta(s,x) = \int_{0}^{\infty} G(t,x) t^{s-1}\,dt$. Applying Lemma \ref{Lemma: mellin} to $tG(t,x)$ yields the analytic continuation and special values
\[
    \zeta(-k,x) = -\frac{B_{k+1}(x)}{k+1} \quad \text{and} \quad L(-k,\chi) = -\sum_{a=1}^N \chi(a) N^{k} \frac{B_{k+1}(a/N)}{k+1},
\]
where the \textit{Bernoulli polynomials} $B_n(x) \in \mathbb{Q}[x]$ are generated by
\begin{equation}\label{eq: hurwitz generating function}
    \sum_{n=0}^{\infty} B_n(x) \frac{t^n}{n!} = \frac{t e^{tx}}{e^{t}-1} = t G(t,1-x) = -t G(-t,x)
\end{equation}
and satisfy $B_n(x) = \sum_{j=0}^n (-1)^j \binom{n}{j} B_j x^{n-j}$ (see \cite[Proposition 4.9]{AIK14}).

\paragraph{Lattices and partial zeta functions.}
Let $F$ be a totally real number field of degree $n = [F:\mathbb{Q}]$, and let $L \subset F \otimes_{\mathbb{Q}} \mathbb{R} \cong \mathbb{R}^n$ be a lattice arising from a fractional ideal of $F$. Denote by $(-)^+$ the subset of totally positive elements (lying in the positive orthant $\mathbb{R}_{>0}^n$). For a nonzero torsion point $\mathbf{t} \in (F \otimes_{\mathbb{Q}} \mathbb{R})/L$, let $\Gamma = \mathcal{O}_{\mathbf{t}}^{\times,+}$ be the group of totally positive units of $\mathcal{O}_F$ stabilizing $\mathbf{t}$ modulo $L$.

\begin{definition}
The \textit{lattice zeta function} associated to $L$ and $\mathbf{t}$ is defined for $\Re(s)>1$ by
\begin{equation}\label{eq lattice zeta}
    \zeta(\mathbf{t}, L, s) := \sum_{\alpha \in (\mathbf{t} + L)^+/\Gamma} N(\alpha)^{-s}.
\end{equation}
\end{definition}

\begin{definition}
Let $\mathfrak{f}$ and $\mathfrak{b}$ be nonzero coprime integral ideals of $\mathcal{O}_F$. The \textit{partial zeta function} of the narrow ray class of $\mathfrak{b}$ modulo $\mathfrak{f}$ is defined for $\Re(s)>1$ by
\begin{equation}\label{eq partial zeta}
    \zeta(\mathfrak{b}, \mathfrak{f}, s) := \sum_{\mathfrak{g} \sim_\mathfrak{f} \mathfrak{b}} N(\mathfrak{g})^{-s},
\end{equation}
where the sum runs over all integral ideals $\mathfrak{g} \subset \mathcal{O}_F$ in the narrow ray class of $\mathfrak{b}$ modulo $\mathfrak{f}$.
\end{definition}

Every such ideal $\mathfrak{g}$ has the form $\alpha\mathfrak{b}$ with $\alpha \in (1+\mathfrak{f}\mathfrak{b}^{-1})^+ = (1+L)^+$ for $L = \mathfrak{f}\mathfrak{b}^{-1}$, unique up to the action of $\Gamma = \mathcal{O}_1^{\times,+}$, which gives the relation
\begin{equation}\label{lattice zeta and partial}
    \zeta(\mathfrak{b}, \mathfrak{f}, s) = N(\mathfrak{b})^{-s} \zeta(1, \mathfrak{f}\mathfrak{b}^{-1}, s).
\end{equation}
When $F = \mathbb{Q}$, this recovers the Hurwitz zeta function. In higher dimensions, Siegel and Klingen (see \cite{Sie70}) proved that $\zeta(\mathfrak{b}, \mathfrak{f}, -k) \in \mathbb{Q}$ for all integers $k \geq 0$ (Corollary \ref{cor: klingen siegel}). A geometric approach to this rationality result was later given by Shintani \cite{Shi76} by decomposing partial zeta functions into finite linear combinations of cone zeta functions via polyhedral cone decompositions of the totally positive orthant.

\paragraph{Shintani zeta functions.}
Let $A = (a_{ij})$ be a complex $r \times m$ matrix with $\Re(a_{ij}) > 0$ for all $i,j$; let $\chi = (\chi_1, \dots, \chi_r) \in \mathbb{C}^r$ with $|\chi_i| \leq 1$; and let $\mathbf{x} = (x_1, \dots, x_r) \in \mathbb{R}^r$ with $0 \leq x_i \leq 1$ and $\mathbf{x} \neq 0$. Define linear forms $L_i$ on $\mathbb{C}^m$ and $L_j^*$ on $\mathbb{C}^r$ by
\[
    L_i(\mathbf{z}) := \sum_{k=1}^m a_{ik}z_k, \qquad L_j^*(\mathbf{w}) := \sum_{k=1}^r a_{kj}w_k.
\]

\begin{definition}\label{def shintani zeta function}
The \textit{Shintani zeta function} of $\mathbf{s} = (s_1, \dots, s_m) \in \mathbb{C}^m$ is defined for $\Re(s_j) > \frac{r}{m}$ by
\[
    \zeta(\mathbf{s}, A, \mathbf{x}, \chi) := \sum_{\mathbf{n} \in \mathbb{N}^r} \chi^{\mathbf{n}} \prod_{j=1}^m L_j^*(\mathbf{n}+\mathbf{x})^{-s_j}.
\]
\end{definition}

Introducing the generating function for $\mathbf{t} \in \mathbb{R}_{>0}^m$,
\[
    G(\mathbf{t}, A, \mathbf{x}, \chi) := \sum_{\mathbf{n} \in \mathbb{N}^r} \chi^{\mathbf{n}} e^{-\sum_{j=1}^m L_j^*(\mathbf{n}+\mathbf{x})t_j} = \prod_{i=1}^r \frac{e^{-x_i L_i(\mathbf{t})}}{1 - \chi_i e^{-L_i(\mathbf{t})}},
\]
we obtain the multidimensional Mellin transform representation
\begin{equation}\label{eq: integral shintani}
    \zeta(\mathbf{s}, A, \mathbf{x}, \chi) = \int_0^\infty \cdots \int_0^\infty G(\mathbf{t}, A, \mathbf{x}, \chi) \prod_{j=1}^m \frac{t_j^{s_j-1}}{\Gamma(s_j)}\,dt_j.
\end{equation}

\begin{lemma}\label{Lemma: gen Mellin}
Suppose $f \in C^{\infty}(\mathbb{R}_{\geq 0}^m)$ and all of its partial derivatives decrease rapidly at infinity. Then
\[
    M(f,\mathbf{s}) := \int_0^\infty \cdots \int_0^\infty f(\mathbf{t}) \prod_{j=1}^m \frac{t_j^{s_j-1}}{\Gamma(s_j)}\,dt_j,
\]
defined for $\Re(s_j)>0$, admits an analytic continuation to $\mathbb{C}^m$, and for $\mathbf{k} = (k_1, \dots, k_m) \in \mathbb{N}^m$,
\[
    M(f,-\mathbf{k}) = \prod_{j=1}^m \left(-\frac{\partial}{\partial t_j}\right)^{k_j} f(\mathbf{t}) \Big|_{\mathbf{t}=0}.
\]
\end{lemma}

\begin{proof}
This follows by iterating Lemma \ref{Lemma: mellin} in each variable (see also \cite[Lemme 3.2]{Col88}).
\end{proof}

When $\chi_i \neq 1$ for all $i$, $G(\mathbf{t}, A, \mathbf{x}, \chi)$ is smooth at $\mathbf{t}=0$ and Lemma \ref{Lemma: gen Mellin} directly yields its special values at $-\mathbf{k}$. When $\chi_i = 1$, however, $G(\mathbf{t}, A, \mathbf{x}, 1)$ has poles along the hyperplanes $L_i(\mathbf{t}) = 0$, requiring either a sectorial decomposition of $\mathbb{R}_{>0}^m$ as in \cite{Shi76} or an algebraic Laurent series formalism as developed in Section \ref{section Field of Coefficients} below.

\section{The Topological Polylogarithm and Residues of Currents}\label{sec:polylog}

\subsection{Setup}\label{section setup}

Throughout this paper, we follow the convention ``right action on spaces, left action on cohomology.'' The identifications below offer a good dictionary when comparing different approaches throughout the literature (see section \ref{sec:literature_comparison} for detailed comparisons).

\paragraph{Orientations, bases, and identifications.}
Let $F$ be a totally real number field of degree $n = [F:\mathbb{Q}]$, let $F^{\mathrm{Gal}} \subset \mathbb{R}$ be the Galois closure of $F$ over $\mathbb{Q}$ inside $\mathbb{R}$, and let $\mathfrak{f}$ and $\mathfrak{b}$ be nonzero integral ideals of $\mathcal{O}_F$. Fix an ordering $J_1, \dots, J_n: F \hookrightarrow F^{\mathrm{Gal}} \subset \mathbb{R}$ of the real embeddings of $F$. This induces an isomorphism of $\mathbb{R}$-algebras and $\mathbb{R}$-vector spaces:
\[
\begin{tikzcd}[row sep=tiny]
    F \otimes_{\mathbb{Q}} \mathbb{R} \arrow[rr, "J", "\cong"'] && \prod_{i=1}^n \mathbb{R} = \mathbb{R}^n \\
    x \otimes r \arrow[rr, maps to] && r\bigl(J_1(x), \dots, J_n(x)\bigr).
\end{tikzcd}
\]
Let $\mathbf{w} = (w_1, \dots, w_n)$ be a $\mathbb{Z}$-basis of the fractional ideal $\mathfrak{f}\mathfrak{b}^{-1}$ (and hence a $\mathbb{Q}$-basis of $F$). We obtain the following identifications of $\mathbb{R}$-vector spaces and lattices:
\begin{equation}\label{identification}
\begin{tikzcd}[row sep=tiny]
    \mathbb{Q}^n \otimes_{\mathbb{Q}} \mathbb{R} \arrow[rr, "\cdot \mathbf{w}", "\cong"'] && F \otimes_{\mathbb{Q}} \mathbb{R} \arrow[rr, "J", "\cong"'] && \mathbb{R}^n =: V \\
    \mathbb{Z}^n \otimes 1 \arrow[rr, maps to] && \mathfrak{f}\mathfrak{b}^{-1} \otimes 1 \arrow[rr, maps to] && L
\end{tikzcd}
\end{equation}
Under \eqref{identification}, pulling back the norm $N$ via $J(\mathbf{w})$ gives $\zeta(1, \mathfrak{f}\mathfrak{b}^{-1}, s) = \zeta(1, L, s)$ and $\zeta(\mathfrak{b}, \mathfrak{f}, s) = N(\mathfrak{b})^{-s}\zeta(1, L, s)$.

\paragraph{The norm-one hypersurface.}
Write $V = \mathbb{R}^n$ and $V_{>0} = \mathbb{R}_{>0}^n$, and set $(F \otimes_{\mathbb{Q}} \mathbb{R})_{>0} := J^{-1}(V_{>0})$. As multiplicative topological groups, we have a split short exact sequence
\begin{equation}\label{split ses}
\begin{tikzcd}
    1 \arrow[r] & V^1 \arrow[r] & V_{>0} \arrow[r, "N"] \arrow[l, "{p_1}"', bend right=25] & \mathbb{R}_{>0} \arrow[r] & 1,
\end{tikzcd}
\end{equation}
where $N(v) = v_1 \cdots v_n$ is the norm map, $V^1 := \{q \in V_{>0} \mid N(q) = 1\}$ is the norm-one hypersurface, and $p_1(v) := N(v)^{-1/n} v$ is the radial projection. This induces the explicit diffeomorphism
\begin{equation}\label{eq:V1_splitting}
    V_{>0} \xrightarrow{\sim} V^1 \times \mathbb{R}_{>0}, \qquad v \mapsto \bigl(N(v)^{-1/n}v,\, N(v)^{1/n}\bigr),
\end{equation}
with inverse $(q, t) \mapsto tq$.

Let $\Gamma \subset \mathcal{O}_F^{\times,+}$ be a subgroup of finite index of the group of totally positive units. Then the action of $\Gamma$ on $\mathfrak{f}\mathfrak{b}^{-1}$ preserves $L$, $V_{>0}$, and $V^1$, and is given by:
\begin{equation}
\begin{tikzcd}
F \otimes_{\mathbb{Q}} \mathbb{R} \arrow[loop above, "{F^\times}"] \arrow[rr, "J", "\cong"'] && \mathbb{R}^n \arrow[loop above, "{\diag(J(F^\times))}"] &&
\mathfrak{f}\mathfrak{b}^{-1} \arrow[loop above, "{\Gamma}"] \arrow[rr, "J", "\cong"'] && L \arrow[loop above, "{\diag(J(\Gamma))}"],
\end{tikzcd}
\end{equation}
making the identifications in \ref{identification} equivariant. To harmonize left and right actions of the abelian group $\Gamma$ with \cite[(3.12)]{BKL18}, we write $v \cdot \gamma := \gamma^{-1} v$ for the right action of $\gamma \in \Gamma$ on $v \in V$.

\subsection{The geometric picture}\label{section geometric picture}

We now develop the theory of topological polylogarithms from \cite{BKL18} and \cite{Gra16} within our framework.

\paragraph{Torus families and the logarithm sheaf.}\label{paragraph log sheaf}
Let $V, L, \Gamma$ be as in Section \ref{section setup}, and let $\pi: V \to T := V/L$ be the quotient map to the compact real $n$-torus $T$. Setting $\mathcal{S} := V_{>0}/\Gamma$, the action of $\Gamma$ on $L$ defines a family of real tori $(V_{>0} \times T)/\Gamma = (V_{>0} \times V)/(L \rtimes \Gamma) \to \mathcal{S}$ over $\mathcal{S}$.

Let $A$ be a Noetherian commutative ring and let $A[L]$ be the group ring of $L$ with canonical basis $\delta_l$ given by the universal group homomorphism $\delta: L \rightarrow A[L]^\times$, $l \mapsto \delta_l$, so that $l \in L$ acts on $A[L]$ via $\delta_l$. Let
\[
    R := R_A := \varprojlim_k A[L]/J^k
\]
be the completion of $A[L]$ with respect to the augmentation ideal $J := \ker(\aug: A[L] \to A,\, \delta_l \mapsto 1)$, with $I := JR$. The group $L = \pi_1(T)$ acts on $R$ via $l \cdot r = \delta_l r$, and the action of $\Gamma$ on $L$ extends by continuity to $R$.

\begin{definition}\label{def log sheaf}
Equip $V \times R$ with the right $L$-action $(v,r) \cdot l := (v+l, \delta_{-l}r)$. The \textit{logarithm sheaf} $\mathcal{L}\mathrm{og} = \mathcal{L}\mathrm{og}_T$ on $T$ is the locally constant sheaf of $R$-modules defined as the sheaf of sections of the quotient $\tilde{\pi}: (V \times R)/L \to T$. Equivalently,
\[
    \mathcal{L}\mathrm{og} := \underline{R} \otimes_{\underline{A[L]}} \pi_!\underline{A}.
\]
\end{definition}

The right action of $\Gamma$ on $(V \times R)/L$ given by $((v,r), \gamma) \mapsto (\gamma^{-1}v, \gamma^{-1}(r))$ makes $\mathcal{L}\mathrm{og}$ a $\Gamma$-equivariant local system on $T$. 

\paragraph{The positive orthant sphere $X$ and diagram of spaces.}
To perform explicit analytic computations more easily, we consider the unit sphere restricted to the positive orthant:

\begin{definition}
Let $\mathbb{S}^{n-1} \subset V$ be the unit sphere centered at the origin with respect to the Euclidean norm $\Vert x \Vert = \langle x, x \rangle^{1/2}$ induced by $J$. We define the auxiliary space
\[
    X := \mathbb{S}^{n-1} \cap V_{>0}.
\]
\end{definition}

The radial projections $p_X: V_{>0} \to X$, $x \mapsto \frac{x}{\Vert x \Vert}$, and $p_1: V_{>0} \to V^1$, $x \mapsto N(x)^{-1/n}x$, restrict to mutually inverse diffeomorphisms
\[
    \psi := p_X|_{V^1}: V^1 \xrightarrow{\sim} X, \qquad \psi^{-1} = p_1|_X: X \xrightarrow{\sim} V^1.
\]
Through $\psi$, the space $X$ inherits the $\Gamma$-action from $V^1$. Fixing a nonzero torsion point $\mathbf{t} \in T^{(A)}$ stabilized by $\Gamma$ (which defines a constant torsion section), we have the commutative diagram of spaces:
\begin{equation}\label{diagram of spaces}
\begin{tikzcd}
    {\faktor{V_{>0} \times V}{L \rtimes \Gamma}} \arrow[rr, "{p_1}"] \arrow[d, equals] && {\faktor{V^1 \times V}{L \rtimes \Gamma}} \arrow[rr, "\psi", "\cong"'] \arrow[d, equals] && {\faktor{X \times V}{L \rtimes \Gamma}} \arrow[d, equals] \\
    {\faktor{V_{>0} \times T}{\Gamma}} \arrow[rr, "{p_1}"] \arrow[dd] && {\faktor{V^1 \times T}{\Gamma}} \arrow[rr, "\psi", "\cong"'] \arrow[dd] && {\faktor{X \times T}{\Gamma}} \arrow[dd] \\
    \\
    \mathcal{S} \arrow[rr, "{p_1}"] && {\mathcal{S}^1} \arrow[rr, "\psi", "\cong"'] \arrow[uu, "\mathbf{t}"', bend right=25] && {\faktor{X}{\Gamma}} \arrow[uu, "\mathbf{t}"', bend right=25].
\end{tikzcd}
\end{equation}
Since $X \cong V^1 \cong \mathbb{R}^{n-1}$ is contractible, for the punctured torus $\overset{\circ}{T} := T \setminus \{0\}$ (and $\overset{\circ}{V} := V \setminus L$) we have canonically
\[
    H^{n-1}(\overset{\circ}{T}, \Gamma; \mathcal{L}\mathrm{og}) \cong H^{n-1}\left(\faktor{X \times \overset{\circ}{T}}{\Gamma}, \mathcal{L}\mathrm{og}\right) \xrightarrow{\psi^*} H^{n-1}\left(\faktor{V^1 \times \overset{\circ}{T}}{\Gamma}, \mathcal{L}\mathrm{og}\right) \xrightarrow{\mathbf{t}^*} H^{n-1}(\mathcal{S}^1, \mathbf{t}^*\mathcal{L}\mathrm{og}).
\]

\subsection{Cohomology, localization, and residues of currents}\label{sec:currents}

We now review the relation between sheaf cohomology with supports, relative de Rham cohomology, and currents on smooth manifolds (following \cite{Ive86, KS90, dR84, GH94}), and prove that the topological residue is computed by the exterior derivative of currents.

\paragraph{Sheaves on manifolds and distinguished triangles.}
Let $X$ be an oriented smooth manifold of dimension $n$, let $i: Z \hookrightarrow X$ be a closed oriented smooth submanifold of codimension $c = \dim X - \dim Z$, and let $j: U := X \setminus Z \hookrightarrow X$ be the open complement. In the derived category $\mathcal{D}(\mathrm{Sh}(X))$, the pairs of adjoint functors $(j_!, j^* = j^!)$, $(j^*, Rj_*)$, $(i^*, i_* = i_!)$, and $(i_*, Ri^!)$ fit into the dual distinguished triangles
\begin{equation}\tag{$\bigtriangleup_1$}\label{triangle1}
    j_! j^* \longrightarrow 1 \longrightarrow i_* i^* \longrightarrow [+1],
\end{equation}
\begin{equation}\tag{$\bigtriangleup_2$}\label{triangle2}
    i_* Ri^! \longrightarrow 1 \longrightarrow Rj_* j^* \longrightarrow [+1].
\end{equation}
For any sheaf $\mathcal{F} \in \mathrm{Sh}(X)$, \eqref{triangle2} induces the long exact localization sequence in cohomology with supports in $Z$:
\[
    \cdots \longrightarrow H^{p-1}(X, \mathcal{F}) \xrightarrow{j^*} H^{p-1}(U, \mathcal{F}|_U) \xrightarrow{\res} H^p_Z(X, \mathcal{F}) \longrightarrow H^p(X, \mathcal{F}) \xrightarrow{j^*} H^p(U, \mathcal{F}|_U) \longrightarrow \cdots,
\]
while applying $Rf_!$ ($f: X \to \{\mathrm{pt}\}$) to \eqref{triangle1} yields the long exact sequence in compactly supported cohomology:
\[
    \cdots \longrightarrow H^q_c(U, \mathcal{F}|_U) \xrightarrow{j_!} H^q_c(X, \mathcal{F}) \xrightarrow{i^*} H^q_c(Z, \mathcal{F}|_Z) \longrightarrow \cdots.
\]

\paragraph{Differential forms, currents, and the relative de Rham cocone.}
Let $\Omega_X^\bullet$ be the de Rham complex of smooth $\mathbb{C}$-valued differential forms on $X$, which is a fine (hence $c$-soft) resolution of $\underline{\mathbb{C}}_X$. The sheaf of $p$-currents $\widehat{\Omega}_X^p$ assigns to each open $W \subset X$ the continuous dual $\widehat{\Omega}_X^p(W) := \Gamma_c(W, \Omega_X^{n-p})^\vee$ equipped with the differential
\[
    d: \widehat{\Omega}_X^p \longrightarrow \widehat{\Omega}_X^{p+1}, \qquad (d\mathcal{T})(\eta) := (-1)^{p+1}\mathcal{T}(d\eta) \quad \bigl(\eta \in \Omega_c^{n-p-1}(W)\bigr).
\]
On the oriented manifold $X$, the inclusion of smooth forms into currents,
\[
    \mathcal{T}_-: \Omega_X^p(X) \hookrightarrow \widehat{\Omega}_X^p(X), \qquad \omega \mapsto \mathcal{T}_\omega := \left(\eta \mapsto \int_X \omega \wedge \eta\right),
\]
is a cochain map (by Stokes' theorem) and a quasi-isomorphism (\cite[Chapter III]{dR84}).

Since $\Omega_X^\bullet$ and $\widehat{\Omega}_X^\bullet$ are not flasque, to compute the topological residue $\res: H^{p-1}(U, \underline{\mathbb{C}}) \to H^p_Z(X, \underline{\mathbb{C}})$ via differential forms and currents, we must first pass to the cocone $C(j^*)$ of the restriction $j^*$, which shifts the distinguished triangle \eqref{triangle2} in the homotopy category to
\begin{equation}\tag{$\bigtriangleup_{C(j^*)}$}\label{cocone}
    j_* j^*[-1] \xrightarrow{-\res} C(j^*) \longrightarrow 1 \longrightarrow j_* j^*.
\end{equation}
Applied to $\Omega_X^\bullet$, the cocone complex $C(j^*\Omega_X^\bullet)$ is given in degree $p$ by
\[
    C(j^*\Omega_X^\bullet)^p := \Omega_U^{p-1}(U) \oplus \Omega_X^p(X), \qquad d(\theta, \omega) := (j^*\omega - d\theta,\, d\omega),
\]
and computes $H^p_Z(X, \underline{\mathbb{C}})$. In this complex, a closed form $\theta \in \Omega_U^{p-1}(U)$ maps under $-\res$ to $[(\theta, 0)] \in H^p(C(j^*\Omega_X^\bullet))$.

\paragraph{Residue of currents.}
Recall that the restriction of currents $j^*: \widehat{\Omega}_X^p(X) \to \widehat{\Omega}_U^p(U)$ is the transpose of the extension-by-zero inclusion $j_!: \Omega_c^{n-p}(U) \hookrightarrow \Omega_c^{n-p}(X)$. By definition of the support of a current, the kernel of $j^*$ is the space of $p$-currents on $X$ supported in $Z$:
\[
    \Gamma_Z(X, \widehat{\Omega}_X^p) := \ker\left(j^*: \widehat{\Omega}_X^p(X) \to \widehat{\Omega}_U^p(U)\right) = \left(\Omega_c^{n-p}(X) / \Omega_c^{n-p}(U)\right)^\vee.
\]
On the other hand, since $\dim Z = n - c$ and the pullback of test forms $i^*: \Omega_c^{n-p}(X) \twoheadrightarrow \Omega_c^{n-p}(Z)$ is surjective (and vanishes on $\Omega_c^{n-p}(U)$), its transpose defines the pushforward of currents from $Z$ to $X$:
\[
    i_* := (i^*)^\vee: \widehat{\Omega}_Z^{p-c}(Z) \hookrightarrow \Gamma_Z(X, \widehat{\Omega}_X^p) \subset \widehat{\Omega}_X^p(X), \qquad (i_*\mathcal{T}_Z)(\eta) := \mathcal{T}_Z(i^*\eta)
\]
for $\eta \in \Omega_c^{n-p}(X)$. Equipping the shifted complex $\widehat{\Omega}_Z^\bullet(Z)[-c]$ with the standard shifted differential $d_{\widehat{\Omega}_Z[-c]} := (-1)^c d_Z$, the pushforward $i_*: \widehat{\Omega}_Z^\bullet(Z)[-c] \hookrightarrow \Gamma_Z(X, \widehat{\Omega}_X^\bullet)$ is a cochain map: for $\mathcal{T}_Z \in \widehat{\Omega}_Z^{p-c}(Z)$ and $\eta \in \Omega_c^{n-p-1}(X)$,
\[
    \bigl(d_X(i_*\mathcal{T}_Z)\bigr)(\eta) = (-1)^{p+1} \mathcal{T}_Z(i^* d_X\eta) = (-1)^{p+1} \mathcal{T}_Z(d_Z i^*\eta) = (-1)^c \bigl(i_*(d_Z\mathcal{T}_Z)\bigr)(\eta).
\]
Furthermore, since $j^*\mathcal{T} = 0$ for every $\mathcal{T} \in \Gamma_Z(X, \widehat{\Omega}_X^p)$, the inclusion into the second factor of the current cocone $C(j^*\widehat{\Omega}_X^\bullet)^p = \widehat{\Omega}_U^{p-1}(U) \oplus \widehat{\Omega}_X^p(X)$,
\[
    \iota: \Gamma_Z(X, \widehat{\Omega}_X^\bullet) \hookrightarrow C(j^*\widehat{\Omega}_X^\bullet), \qquad \mathcal{T} \mapsto (0, \mathcal{T}),
\]
satisfies $d(0, \mathcal{T}) = (j^*\mathcal{T} - d(0), d\mathcal{T}) = (0, d\mathcal{T}) = \iota(d\mathcal{T})$ and hence is also a cochain map.

\begin{theorem}\label{residue currents}
Let $i: Z \hookrightarrow X$ be a closed oriented smooth submanifold of codimension $c$ in an oriented $n$-dimensional manifold $X$, and let $j: U = X \setminus Z \hookrightarrow X$ be its open complement. Then the natural cochain maps
\[
    \Omega_Z^\bullet(Z)[-c] \xrightarrow{\mathcal{T}_{-,Z}} \widehat{\Omega}_Z^\bullet(Z)[-c] \xrightarrow{i_*} \Gamma_Z(X, \widehat{\Omega}_X^\bullet) \xrightarrow{\iota} C(j^*\widehat{\Omega}_X^\bullet) \xleftarrow{\mathcal{T}_-} C(j^*\Omega_X^\bullet)
\]
are all quasi-isomorphisms. Consequently, they induce canonical isomorphisms in cohomology:
\[
    H^{p-c}(\Omega_Z^\bullet(Z)) \xrightarrow{\sim} H^{p-c}(\widehat{\Omega}_Z^\bullet(Z)) \xrightarrow[\sim]{i_*} H^p(\Gamma_Z(X, \widehat{\Omega}_X^\bullet)) \xrightarrow[\sim]{\iota} H^p(C(j^*\widehat{\Omega}_X^\bullet)) \xleftarrow{\sim} H^p(C(j^*\Omega_X^\bullet)).
\]
In particular, every cohomology class in $H^p(\Gamma_Z(X, \widehat{\Omega}_X^\bullet))$ is represented by the pushforward $i_*\mathcal{T}_{\alpha,Z}$ of a smooth closed $(p-c)$-form $\alpha \in \Omega_Z^{p-c}(Z)$ on $Z$.
\end{theorem}

\begin{proof}
First, on the oriented manifold $Z$, the inclusion of smooth forms into currents $\mathcal{T}_{-,Z}: \Omega_Z^\bullet(Z)[-c] \to \widehat{\Omega}_Z^\bullet(Z)[-c]$ is a quasi-isomorphism by de Rham's theorem (\cite[Chapter III]{dR84}). Similarly, since $\Omega_X^\bullet(X) \to \widehat{\Omega}_X^\bullet(X)$ and $\Omega_U^\bullet(U) \to \widehat{\Omega}_U^\bullet(U)$ are quasi-isomorphisms commuting with $j^*$, the Five Lemma applied to the long exact cohomology sequences of the cocone triangles shows that $\mathcal{T}_-: C(j^*\Omega_X^\bullet) \to C(j^*\widehat{\Omega}_X^\bullet)$, $(\theta, \omega) \mapsto (\mathcal{T}_{\theta,U}, \mathcal{T}_{\omega,X})$, is a quasi-isomorphism.

Second, we show that $i_*: \widehat{\Omega}_Z^\bullet(Z)[-c] \hookrightarrow \Gamma_Z(X, \widehat{\Omega}_X^\bullet)$ is a cochain homotopy equivalence (and hence a quasi-isomorphism). Because the sheaves $\Omega_X^{n-\bullet}$ are $c$-soft, the sequence of compactly supported forms from $(Rf_!\bigtriangleup_1)$ is exact:
\[
    0 \to \Omega_c^{n-\bullet}(U) \xrightarrow{j_!} \Omega_c^{n-\bullet}(X) \xrightarrow{q_Z} \Omega_c^{n-\bullet}(X) / \Omega_c^{n-\bullet}(U) \to 0,
\]
where the quotient $\Omega_c^{n-p}(X) / \Omega_c^{n-p}(U) \cong \Gamma_c(Z, i^*\Omega_X^{n-p})$ is the space of germs along $Z$ of compactly supported $(n-p)$-forms on $X$. Choose a tubular neighbourhood $\pi_Z: N \to Z$ of $Z$ in $X$ and a smooth cutoff function $\chi \in C^\infty(X)$ supported in $N$ such that $\supp(\chi)$ is proper over $Z$ and $\chi \equiv 1$ in an open neighbourhood $V \subset N$ of $Z$. Then every $\eta \in \Omega_c^{n-p}(X)$ decomposes as $\eta = \chi\eta + (1-\chi)\eta$, where $(1-\chi)\eta \in \Omega_c^{n-p}(U)$ vanishes on $V$ and $\chi\eta$ is supported in $N$. The map $s: \Omega_c^{n-\bullet}(Z) \to \Omega_c^{n-\bullet}(X)/\Omega_c^{n-\bullet}(U)$ given by $\xi \mapsto [\chi\,\pi_Z^*\xi]$ satisfies $i^* \circ s = \id$, and by the Poincar\'e lemma along the fibres of the vector bundle $\pi_Z: N \to Z$, $s \circ i^*$ is continuously cochain homotopic to the identity on $\Omega_c^{n-\bullet}(X)/\Omega_c^{n-\bullet}(U)$. Passing to continuous duals, the transpose of $i^*: \Omega_c^{n-\bullet}(X)/\Omega_c^{n-\bullet}(U) \to \Omega_c^{n-\bullet}(Z)$ is precisely
\[
    i_*: \widehat{\Omega}_Z^\bullet(Z)[-c] \to \left(\Omega_c^{n-\bullet}(X)/\Omega_c^{n-\bullet}(U)\right)^\vee = \Gamma_Z(X, \widehat{\Omega}_X^\bullet),
\]
which is therefore a cochain homotopy equivalence with homotopy inverse $\mathcal{T} \mapsto \pi_{Z,*}(\chi\mathcal{T})$, where $\pi_{Z,*}(\chi\mathcal{T})(\xi) := \mathcal{T}(\chi\,\pi_Z^*\xi)$ for $\xi \in \Omega_c^{n-p}(Z)$.

Third, to see that $\iota: \Gamma_Z(X, \widehat{\Omega}_X^\bullet) \to C(j^*\widehat{\Omega}_X^\bullet)$ is a quasi-isomorphism, compare the long exact cohomology sequence of the cocone $C(j^*\widehat{\Omega}_X^\bullet)$ with the $\mathbb{C}$-linear dual of the long exact cohomology sequence of $0 \to \Omega_c^{n-\bullet}(U) \to \Omega_c^{n-\bullet}(X) \to \Omega_c^{n-\bullet}(X)/\Omega_c^{n-\bullet}(U) \to 0$:
\[
\begin{tikzcd}[column sep=small]
    \cdots \arrow[r] & H^{p-1}(\widehat{\Omega}_X^\bullet) \arrow[r, "j^*"] \arrow[d, "\wr"] & H^{p-1}(\widehat{\Omega}_U^\bullet) \arrow[r] \arrow[d, "\wr"] & H^p(C(j^*\widehat{\Omega}_X^\bullet)) \arrow[r] & H^p(\widehat{\Omega}_X^\bullet) \arrow[r] \arrow[d, "\wr"] & \cdots \\
    \cdots \arrow[r] & H_c^{n-p+1}(X)^\vee \arrow[r, "j_!^\vee"] & H_c^{n-p+1}(U)^\vee \arrow[r] & H^p(\Gamma_Z(X, \widehat{\Omega}_X^\bullet)) \arrow[u, "\iota"] \arrow[r] & H_c^{n-p}(X)^\vee \arrow[r] & \cdots
\end{tikzcd}
\]
where the bottom row uses the isomorphism $H^p(\Gamma_Z(X, \widehat{\Omega}_X^\bullet)) \cong H^{p-c}(\widehat{\Omega}_Z^\bullet(Z)) \cong H_c^{n-p}(Z)^\vee \cong H^{n-p}(\Omega_c^\bullet(X)/\Omega_c^\bullet(U))^\vee$ established above, and the vertical maps on $X$ and $U$ are isomorphisms by de Rham's duality theorem for currents (\cite[Chapter IV §22-23]{dR84}). By the Five Lemma, $\iota$ is a quasi-isomorphism.
\end{proof}

\begin{corollary}[Explicit residue formula for locally integrable forms]\label{cor:explicit_residue}
Suppose a cohomology class $[(\theta, \omega)] \in H^p(C(j^*\Omega_X^\bullet))$ is represented by $\theta \in \Omega_U^{p-1}(U)$ and $\omega \in \Omega_X^p(X)$ (satisfying $j^*\omega = d\theta$ on $U$ and $d\omega = 0$ on $X$) such that $\theta$ is locally integrable on $X$ ($\theta \in L^1_{\mathrm{loc}}(X)$), or more generally such that the current $\mathcal{T}_{\theta,U} \in \widehat{\Omega}_U^{p-1}(U)$ extends to a current $\mathcal{T}_\theta \in \widehat{\Omega}_X^{p-1}(X)$ on $X$ with $j^*\mathcal{T}_\theta = \mathcal{T}_{\theta,U}$. Then:
\begin{enumerate}
    \item The current $\mathcal{T}_\omega - d\mathcal{T}_\theta \in \widehat{\Omega}_X^p(X)$ is closed and supported in $Z$, i.e., $\mathcal{T}_\omega - d\mathcal{T}_\theta \in \Gamma_Z(X, \widehat{\Omega}_X^p)$.
    \item In $H^p(C(j^*\widehat{\Omega}_X^\bullet))$, we have
    \[
        [(\mathcal{T}_{\theta,U}, \mathcal{T}_\omega)] = \iota\bigl([\mathcal{T}_\omega - d\mathcal{T}_\theta]\bigr).
    \]
    \item Under the topological residue map $\res: H^{p-1}(\Omega_U^\bullet(U)) \to H^p(C(j^*\Omega_X^\bullet)) \cong H^p(\Gamma_Z(X, \widehat{\Omega}_X^\bullet))$ (which sends a closed form $[\theta]$ to $-[(\theta, 0)]$ by $(\bigtriangleup_{C(j^*)})$), if $\theta \in \Omega_U^{p-1}(U) \cap L^1_{\mathrm{loc}}(X)$ is closed on $U$ ($d\theta = 0$), then
    \[
        \res([\theta]) = [d\mathcal{T}_\theta] \in H^p(\Gamma_Z(X, \widehat{\Omega}_X^\bullet)) \cong H^{p-c}(\widehat{\Omega}_Z^\bullet(Z)).
    \]
    More generally, if $d\theta = j^*\omega$ for a smooth closed form $\omega \in \Omega_X^p(X)$ (so $\mathcal{T}_{d\theta} = \mathcal{T}_\omega$), then the relative class $-[(\theta, \omega)] \in H^p(C(j^*\Omega_X^\bullet))$ corresponds under $\iota^{-1}$ to $[d\mathcal{T}_\theta - \mathcal{T}_{d\theta}]$.
\end{enumerate}
\end{corollary}

\begin{proof}
Since $\mathcal{T}_\theta \in \widehat{\Omega}_X^{p-1}(X)$ satisfies $j^*\mathcal{T}_\theta = \mathcal{T}_{\theta,U}$, we have $j^*(\mathcal{T}_\omega - d\mathcal{T}_\theta) = \mathcal{T}_{j^*\omega, U} - d_U(j^*\mathcal{T}_\theta) = \mathcal{T}_{d\theta, U} - \mathcal{T}_{d\theta, U} = 0$ on $U$, so $\mathcal{T}_\omega - d\mathcal{T}_\theta \in \Gamma_Z(X, \widehat{\Omega}_X^p)$, and $d(\mathcal{T}_\omega - d\mathcal{T}_\theta) = \mathcal{T}_{d\omega} - 0 = 0$.
Moreover, the element $(0, \mathcal{T}_\theta) \in \widehat{\Omega}_U^{p-2}(U) \oplus \widehat{\Omega}_X^{p-1}(X) = C(j^*\widehat{\Omega}_X^\bullet)^{p-1}$ has coboundary
\[
    d(0, \mathcal{T}_\theta) = (j^*\mathcal{T}_\theta, d\mathcal{T}_\theta) = (\mathcal{T}_{\theta,U}, d\mathcal{T}_\theta) \in C(j^*\widehat{\Omega}_X^\bullet)^p.
\]
Subtracting this coboundary from $(\mathcal{T}_{\theta,U}, \mathcal{T}_\omega)$ in $C(j^*\widehat{\Omega}_X^\bullet)^p$ yields
\[
    (\mathcal{T}_{\theta,U}, \mathcal{T}_\omega) - d(0, \mathcal{T}_\theta) = (0, \mathcal{T}_\omega - d\mathcal{T}_\theta) = \iota(\mathcal{T}_\omega - d\mathcal{T}_\theta),
\]
proving (2). Statement (3) follows immediately from (2) and $(\bigtriangleup_{C(j^*)})$, matching \cite[p.~368]{GH94}.
\end{proof}

\begin{remark}\label{delta distribution}
When $Z = \{p\}$ is a single point in an $n$-dimensional manifold $X$, the codimension is $c = n$. Taking $p = n$, the pushforward $i_*: \widehat{\Omega}_{\{p\}}^0(\{p\}) \hookrightarrow \Gamma_{\{p\}}(X, \widehat{\Omega}_X^n)$ identifies $\mathbb{C} \cong \widehat{\Omega}_{\{p\}}^0(\{p\})$ with the span of the Dirac delta current via $z \mapsto z\,\delta_{\{p\}}$. Thus, for a locally integrable closed $(n-1)$-form $\theta \in \Omega^{n-1}(X \setminus \{p\}) \cap L^1_{\mathrm{loc}}(X)$ (or one with $d\theta = j^*\omega$), the current equation $d\mathcal{T}_\theta - \mathcal{T}_\omega = z\,\delta_{\{p\}}$ on $X$ means precisely that $\res([\theta]) = z \in \mathbb{C} \cong H^n_{\{p\}}(X, \underline{\mathbb{C}})$.
\end{remark}

\subsection{Field of coefficients and the extended logarithm sheaf}\label{section Field of Coefficients}

\paragraph{Coordinates on $R_{\mathbb{K}}$ and orientation.}
Let $A = \mathbb{K}$ be a subfield of $\mathbb{C}$ containing the Galois closure $F^{\mathrm{Gal}} \subset \mathbb{R}$. By \cite[Corollary 3.10]{BKL18}, the exponential map $\delta_l \mapsto \exp(l) = \sum_{k \geq 0} \frac{l^{\otimes k}}{k!}$ induces an isomorphism $\exp^*: R_{\mathbb{K}} \xrightarrow{\sim} \widehat{\Sym}\, L_{\mathbb{K}}$. Together with a choice of $\mathbb{Z}$-basis $\mathbf{w}$ of $L$ and the embedding $J$ of \eqref{identification}, we obtain the commutative diagram of isomorphisms:
\[
\begin{tikzcd}[column sep=small,row sep=small]
    {\mathbb{K}[X_1^{\pm 1}, \dots, X_n^{\pm 1}]} \arrow[rr, hook] \arrow[dd, "\mathbf{w}", "\cong"'] && {\mathbb{K}\llbracket X_1-1, \dots, X_n-1 \rrbracket} \arrow[rr, "\cong", "{\exp^*}"'] \arrow[dd, "\mathbf{w}", "\cong"'] && {\widehat{\Sym}\,\mathbb{Z}^n \otimes \mathbb{K}} \arrow[rr, equals] \arrow[dd, "\mathbf{w}", "\cong"'] && {\mathbb{K}\llbracket \mathbf{z}_{\mathbf{w}} \rrbracket} \arrow[dd, "{J(\mathbf{w})}", "\cong"'] \\
    \\
    {\mathbb{K}[L]} \arrow[rr, hook] && R_{\mathbb{K}} \arrow[rr, "\cong", "{\exp^*}"'] && {\widehat{\Sym}\,L_{\mathbb{K}}} \arrow[rr, "\cong", "J"'] && {\mathbb{K}\llbracket \mathbf{z} \rrbracket},
\end{tikzcd}
\]
where under the bottom row $\delta_l \mapsto e^{\langle l, \mathbf{z} \rangle} = \exp\bigl(\sum_{i=1}^n J_i(l) z_i\bigr) \in \mathbb{K}\llbracket \mathbf{z} \rrbracket$. The choice of $J$ also trivializes the orientation sheaf $\lambda_L := \bigwedge^n L = H_n(L, \mathbb{Z}) \cong H_n(T, \mathbb{Z})$.

\paragraph{Cone-adapted multivariate Laurent series and $\Gamma$-action.}
Because the cone generating functions $\mathcal{G}_{\sigma_{\boldsymbol{\alpha}}}(\mathbf{t}, \mathbf{z})$ of Definition \ref{def: generating} have poles along the hyperplanes $\langle \alpha_j, \mathbf{z} \rangle = 0$, we enlarge $R_{\mathbb{K}} \cong \mathbb{K}\llbracket \mathbf{z} \rrbracket$ to a field of formal Laurent series in which a notion of monomial coefficient is well-defined, following Aparicio Monforte and Kauers \cite{MK13}.

\begin{definition}
Let $\leq_{\mathrm{lex}}$ be the standard \textit{lexicographic order} on $\mathbb{Z}^n \subset \mathbb{R}^n$, defined by $\mathbf{e} \leq_{\mathrm{lex}} \mathbf{e}'$ if $\mathbf{e} = \mathbf{e}'$ or $e'_{i_0} - e_{i_0} > 0$ for $i_0 := \min\{i \mid e'_i - e_i \neq 0\}$. Let $\mathscr{C}_{\mathrm{lex}}$ be the family of all line-free cones in $\mathbb{R}^n$ compatible with $\leq_{\mathrm{lex}}$ (that is, $\sigma \cap (- \sigma) = \{0\}$ and $\mathbf{e} \in \sigma \cap \mathbb{Z}^n \implies 0 \leq_{\mathrm{lex}} \mathbf{e}$).
\end{definition}

Note that every totally positive cone $\sigma \subset V_{>0} \cup \{0\}$ belongs to $\mathscr{C}_{\mathrm{lex}}$.

\begin{proposition}\label{prop:MK13_field}
Let $\mathbb{K}$ be a field of characteristic $0$. For any cone $\sigma \in \mathscr{C}_{\mathrm{lex}}$, define
\[
    \mathbb{K}_\sigma\llbracket \mathbf{z} \rrbracket := \left\{ \sum_{\mathbf{e} \in \sigma \cap \mathbb{Z}^n} a_{\mathbf{e}} \mathbf{z}^{\mathbf{e}} \;\middle|\; a_{\mathbf{e}} \in \mathbb{K} \right\}, \qquad \mathbb{K}_{\leq_{\mathrm{lex}}}\llbracket \mathbf{z} \rrbracket := \bigcup_{\sigma \in \mathscr{C}_{\mathrm{lex}}} \mathbb{K}_\sigma\llbracket \mathbf{z} \rrbracket.
\]
Then $\mathbb{K}_{\leq_{\mathrm{lex}}}\llbracket \mathbf{z} \rrbracket$ is an integral domain, and
\[
    \widetilde{R}_{\mathbb{K}} := \mathbb{K}_{\leq_{\mathrm{lex}}}(\!(\mathbf{z})\!) := \bigcup_{\mathbf{e} \in \mathbb{Z}^n} \mathbf{z}^{\mathbf{e}} \mathbb{K}_{\leq_{\mathrm{lex}}}\llbracket \mathbf{z} \rrbracket = \Frac\bigl(\mathbb{K}_{\leq_{\mathrm{lex}}}\llbracket \mathbf{z} \rrbracket\bigr)
\]
is a field.
\end{proposition}

\begin{proof}
See \cite[Theorem 15]{MK13}.
\end{proof}

When $\mathbb{K} \supset F^{\mathrm{Gal}}$, the action of $\gamma \in \Gamma \subset \mathcal{O}_F^{\times,+}$ on $L_{\mathbb{K}} \cong \mathbb{K}^n$ is diagonal in the coordinates $\mathbf{z} = (z_1, \dots, z_n)$, given by $\gamma \cdot z_i = J_i(\gamma) z_i$. Consequently, $\gamma \in \Gamma$ acts on any series in $\widetilde{R}_{\mathbb{K}}$ by rescaling each monomial coefficient without altering the exponent $\mathbf{e} \in \mathbb{Z}^n$:
\[
    \gamma \cdot \sum_{\mathbf{e} \in \mathbb{Z}^n} a_{\mathbf{e}} z_1^{e_1} \cdots z_n^{e_n} = \sum_{\mathbf{e} \in \mathbb{Z}^n} \left(\prod_{i=1}^n J_i(\gamma)^{e_i}\right) a_{\mathbf{e}} z_1^{e_1} \cdots z_n^{e_n}.
\]
Thus every cone ring $\mathbb{K}_\sigma\llbracket \mathbf{z} \rrbracket$ and the field $\widetilde{R}_{\mathbb{K}}$ are $\Gamma$-invariant. By Dirichlet's Unit Theorem, $\prod_{i=1}^n J_i(\gamma)^{e_i} = 1$ for all $\gamma \in \Gamma$ if and only if $e_1 = \dots = e_n$, so the diagonal norm projection
\begin{equation}\label{eq:rho_N_def}
    \rho_N: \widetilde{R}_{\mathbb{K}} \twoheadrightarrow (\widetilde{R}_{\mathbb{K}})_\Gamma \cong \mathbb{K}(\!(w)\!), \qquad \sum_{\mathbf{e} \in \mathbb{Z}^n} a_{\mathbf{e}} z_1^{e_1} \cdots z_n^{e_n} \mapsto \sum_{k \in \mathbb{Z}} a_{(k, \dots, k)} w^k,
\end{equation}
is a well-defined $\mathbb{K}$-linear projection onto $\Gamma$-coinvariants extending $\rho_\Gamma: R_{\mathbb{K}} \cong \mathbb{K}\llbracket \mathbf{z} \rrbracket \twoheadrightarrow R_{\mathbb{K},\Gamma} \cong \mathbb{K}\llbracket w \rrbracket$ ($w := z_1 \cdots z_n$).

\paragraph{The extended logarithm sheaf and the topological polylogarithm.}
Write $R = R_{\mathbb{K}} \cong \mathbb{K}\llbracket \mathbf{z} \rrbracket$ and $\widetilde{R} = \widetilde{R}_{\mathbb{K}} = \mathbb{K}_{\leq_{\mathrm{lex}}}(\!(\mathbf{z})\!)$. Since the non-negative orthant $\mathbb{R}_{\geq 0}^n$ belongs to $\mathscr{C}_{\mathrm{lex}}$, we have a canonical $\Gamma$-equivariant injective ring homomorphism $R \hookrightarrow \widetilde{R}$, which makes $\widetilde{R}$ an $R$-algebra and a $\mathbb{K}[L]$-module (where $l \in L$ acts by multiplication by $e^{\langle l, \mathbf{z} \rangle} \in R^\times$).

\begin{definition}\label{def:extended_log_sheaf}
We define the \textit{extended logarithm sheaf} $\widetilde{\mathcal{L}\mathrm{og}}$ on $T$ by extension of scalars from $\mathcal{L}\mathrm{og}$ along $R \hookrightarrow \widetilde{R}$:
\[
    \widetilde{\mathcal{L}\mathrm{og}} := \mathcal{L}\mathrm{og} \otimes_{\underline{R}} \underline{\widetilde{R}}.
\]
\end{definition}

\begin{proposition}\label{prop: vanishing coh}
Let $A = \mathbb{K}$ be a subfield of $\mathbb{C}$ containing $F^{\mathrm{Gal}}$, and fix the orientation of $T$ induced by $J$ so that $\lambda_L$ is trivialized. Then
\[
    H^k(T, \widetilde{\mathcal{L}\mathrm{og}}) = 0 \quad \text{for all } k \geq 0.
\]
\end{proposition}

\begin{proof}
We first show that $\widetilde{R}$ is a flat $R$-module. Since $R \cong \mathbb{K}\llbracket \mathbf{z} \rrbracket$ is an integral domain and $\widetilde{R}$ is a field containing $R$, the universal property of the field of fractions $\Frac(R)$ implies that the inclusion $R \hookrightarrow \widetilde{R}$ factors as
\[
    R \hookrightarrow \Frac(R) \hookrightarrow \widetilde{R}.
\]
The localization $\Frac(R)$ is flat over $R$, and $\widetilde{R}$ is a vector space over the field $\Frac(R)$ and hence is free (and flat) over $\Frac(R)$. By transitivity of flatness, $\widetilde{R}$ is flat over $R$.

Since $\widetilde{\mathcal{L}\mathrm{og}} = \mathcal{L}\mathrm{og} \otimes_{\underline{R}} \underline{\widetilde{R}}$ and $\widetilde{R}$ is flat over $R$, we have
\[
    H^k(T, \widetilde{\mathcal{L}\mathrm{og}}) \cong H^k(T, \mathcal{L}\mathrm{og}) \otimes_R \widetilde{R} \quad \text{for all } k \geq 0.
\]
By \cite[Theorem 3.25]{BKL18}, $H^k(T, \mathcal{L}\mathrm{og}) = 0$ for $k \neq n$ and $H^n(T, \mathcal{L}\mathrm{og}) \cong R/IR = \mathbb{K}$. Thus $H^k(T, \widetilde{\mathcal{L}\mathrm{og}}) = 0$ for $k \neq n$, while for $k = n$ we have
\[
    H^n(T, \widetilde{\mathcal{L}\mathrm{og}}) \cong \mathbb{K} \otimes_R \widetilde{R} = 0,
\]
since $\mathbf{z} \in R$ acts trivially on the left but not on the right.

\end{proof}

\begin{corollary}\label{cor: localization coh}
Let $D \subset T$ be a finite non-empty subset. Then $H^i(T \setminus D, \widetilde{\mathcal{L}\mathrm{og}}) = H^i(T \setminus D, \mathcal{L}\mathrm{og}) = 0$ for $i \neq n-1$. For $\mathcal{L}\mathrm{og}$, the localization sequence yields a short exact sequence
\[
    0 \longrightarrow H^{n-1}(T \setminus D, \mathcal{L}\mathrm{og}) \xrightarrow{\res} \mathcal{L}\mathrm{og}|_D \xrightarrow{\sigma_D} \mathbb{K} \longrightarrow 0,
\]
whereas for $\widetilde{\mathcal{L}\mathrm{og}}$, the vanishing of $H^n(T, \widetilde{\mathcal{L}\mathrm{og}}) = 0$ reduces the localization sequence to a canonical isomorphism
\[
    H^{n-1}(T \setminus D, \widetilde{\mathcal{L}\mathrm{og}}) \xrightarrow[\cong]{\res} \widetilde{\mathcal{L}\mathrm{og}}|_D = \bigoplus_{d \in D} \widetilde{\mathcal{L}\mathrm{og}}_d.
\]
\end{corollary}

\begin{proof}
This follows immediately from Proposition \ref{prop: vanishing coh} and the localization sequence for $D \subset T$, using excision and purity on small contractible neighbourhoods $U_d$ of $d \in D$ to identify $H^n_D(T, \widetilde{\mathcal{L}\mathrm{og}}) \cong \bigoplus_{d \in D} H^n_{\{d\}}(U_d, \widetilde{\mathcal{L}\mathrm{og}}|_{U_d}) \cong \widetilde{\mathcal{L}\mathrm{og}}|_D$.
\end{proof}

\begin{definition}\label{def polylog}
Taking $D = \{0\}$, the \textit{topological polylogarithm} $\pol \in H^{n-1}(T \setminus \{0\}, \Gamma; \widetilde{\mathcal{L}\mathrm{og}})$ is the unique class satisfying the residue condition
\[
    \res(\pol) = \delta_0 = 1 \in \widetilde{\mathcal{L}\mathrm{og}}_0 = \widetilde{R}.
\]
\end{definition}

\begin{remark}\label{BL 1 dimensional} The one dimensional case corresponds precisely to the contents of \cite[\S 2.3]{BL94}.
When the dimension $n = 1$, $\overset{\circ}{T} = \mathbb{S}^1 \setminus \{0\}$ and $\widetilde{R}_{\mathbb{R}} = \mathbb{R}(\!(z)\!)$. Under the trivialization $v \mapsto e^{-vz}$, horizontal sections of $\widetilde{\mathcal{L}\mathrm{og}}^\infty$ satisfy $(d - z\,dv)g = 0$, so $g(v) = c(z)e^{vz}$ for $v \in (0,1)$. By Corollary \ref{cor:explicit_residue} and Remark \ref{delta distribution}, for any test function $f \in C^\infty(\mathbb{S}^1)$ (with $f(0) = f(1)$),
\[
    \res(e^{vz})(f) = \int_0^1 e^{vz} f'(v)\,dv - \int_0^1 z e^{vz} f(v)\,dv = \bigl[e^{vz}f(v)\bigr]_0^1 = (e^z - 1)f(0),
\]
or with the orientation convention of \cite[\S 2]{BKL18}, $\res\bigl(\frac{e^{\{v\}z}}{1-e^z}\bigr) = \delta_0$. Thus $\pol(v) = \frac{e^{\{v\}z}}{1-e^z} \in \widetilde{R}_{\mathbb{R}}$ is precisely the generating function of the Hurwitz zeta values $\zeta(-k, v)$ of \ref{eq: hurwitz generating function}.
\end{remark}

\section{Shintani Decompositions and Cocycles}\label{sec:shintani}

\subsection{Shintani cone decompositions}\label{section Shintani cone decompositions}

Throughout this section we follow the conventions regarding Shintani decompositions from \cite{BHY19}. For a comparison with the conventions from \cite{CDG15}, see subsection \ref{subsection:CDG comparison}.

\begin{definition}
A \textit{rational closed polyhedral cone} in $V_{>0} \cup \{0\}$ (which we simply call a \textit{cone}) is a subset of the form
\[
    \sigma_{\boldsymbol{\alpha}} := \{x_1 \alpha_1 + \dots + x_m \alpha_m \mid x_1, \dots, x_m \in \mathbb{R}_{\geq 0}\}
\]
for some tuple $\boldsymbol{\alpha} = (\alpha_1, \dots, \alpha_m) \in L_{>0}^m$, called a generator of $\sigma_{\boldsymbol{\alpha}}$ (with $\sigma = \{0\}$ when $m=0$). The dimension $\dim \sigma$ is the dimension of the $\mathbb{R}$-linear span of $\sigma$. A cone $\sigma_{\boldsymbol{\alpha}}$ is \textit{simplicial} if $\alpha_1, \dots, \alpha_m$ are linearly independent over $\mathbb{R}$; any cone generated by a subset of $\{\alpha_1, \dots, \alpha_m\}$ is called a \textit{face} of $\sigma_{\boldsymbol{\alpha}}$.
\end{definition}

For any subset $R \subset V_{>0}$, we define its \textit{upper closure} by
\[
    \widehat{R} := \{(u_1, \dots, u_n) \in V_{>0} \mid \exists \delta > 0 \text{ such that } \forall \delta' \in (0, \delta),\, (u_1, \dots, u_{n-1}, u_n - \delta') \in R\}.
\]
If $\dim \sigma < n$, then $\widehat{\sigma} = \emptyset$; if $\dim \sigma = n$, $\widehat{\sigma}$ is obtained from $\sigma$ by removing its $(n-1)$-dimensional lower faces with respect to the last coordinate direction $(0, \dots, 0, 1)$.

A \textit{simplicial fan} $\Phi$ is a collection of simplicial cones such that every face of a cone in $\Phi$ belongs to $\Phi$, and the intersection $\sigma \cap \sigma'$ of any two cones in $\Phi$ is a common face of both.

\begin{definition}\label{def: Shintani}
A \textit{Shintani decomposition} of $V_{>0} \cup \{0\}$ (with respect to $\Gamma$) is a simplicial fan $\Phi$ satisfying:
\begin{enumerate}
    \item $V_{>0} \cup \{0\} = \coprod_{\sigma \in \Phi} \sigma^\circ$, where $\sigma^\circ$ denotes the relative interior of $\sigma$.
    \item For all $\sigma \in \Phi$ and $\gamma \in \Gamma$, we have $\sigma \cdot \gamma \in \Phi$.
    \item The quotient $\Phi/\Gamma$ is a finite set.
\end{enumerate}
\end{definition}

\begin{theorem}[Shintani's Unit Theorem]
A Shintani decomposition $\Phi$ always exists. Moreover, denoting by $\Phi_q$ the subset of cones of dimension $q$, the maximal cones $\Phi_n$ give a disjoint decomposition
\begin{equation}\label{eq: upper closure}
    V_{>0} = \coprod_{\sigma \in \Phi_n} \widehat{\sigma}.
\end{equation}
\end{theorem}

\begin{proof}
See \cite[Theorem 1]{Shi76}, \cite[\S 2.7, Theorem 1]{Hid93}, and \cite[Theorem 4.1 and Proposition 5.6]{Yam10}.
\end{proof}

\subsection{Cone zeta functions and generating series}

\begin{definition}\label{def: shintani zeta}
Let $\sigma \subset V_{>0} \cup \{0\}$ be a cone and $\mathbf{t} \in V/L$ be a torsion point. The \textit{cone zeta function} $\zeta_\sigma(\mathbf{t}, \mathbf{s})$ is defined for $\mathbf{s} = (s_1, \dots, s_n) \in \mathbb{C}^n$ with $\Re(s_j) > 1$ by
\begin{equation}\label{eq: Shintani zeta}
    \zeta_\sigma(\mathbf{t}, \mathbf{s}) := \sum_{\alpha \in \widehat{\sigma} \cap (\mathbf{t} + L)} \prod_{j=1}^n \alpha_j^{-s_j},
\end{equation}
where $\alpha = (\alpha_1, \dots, \alpha_n) \in V_{>0}$. Along the diagonal $\mathbf{s} = (s, \dots, s)$, we have
\begin{equation}\label{eq: diagonal}
    \zeta_\sigma(\mathbf{t}, (s, \dots, s)) = \sum_{\alpha \in \widehat{\sigma} \cap (\mathbf{t} + L)} N(\alpha)^{-s}.
\end{equation}
\end{definition}

\begin{proposition}\label{prop: fac_shintani}
For any $n$-dimensional simplicial cone $\sigma = \sigma_{\boldsymbol{\alpha}}$ generated by $\boldsymbol{\alpha} = (\alpha_1, \dots, \alpha_n) \in L_{>0}^n$, the cone zeta function $\zeta_\sigma(\mathbf{t}, \mathbf{s})$ is a finite sum of Shintani zeta functions:
\[
    \zeta_\sigma(\mathbf{t}, \mathbf{s}) = \sum_{\mathbf{x} \in R_{\boldsymbol{\alpha}}^P} \zeta(\mathbf{s}, A_{\boldsymbol{\alpha}}, \mathbf{x}, 1),
\]
where $P_{\boldsymbol{\alpha}} := \{\sum_{i=1}^n x_i \alpha_i \mid 0 \leq x_i \leq 1\}$ is the fundamental parallelepiped of $\boldsymbol{\alpha}$, $R_{\boldsymbol{\alpha}}^P := \{\mathbf{x} \in [0,1]^n \mid \sum_{i=1}^n x_i \alpha_i \in \widehat{P}_{\boldsymbol{\alpha}} \cap (\mathbf{t} + L)\}$, and $A_{\boldsymbol{\alpha}}$ is the $n \times n$ matrix whose $i$-th row is $\alpha_i = (J_1(\alpha_i), \dots, J_n(\alpha_i))$.
\end{proposition}

\begin{proof}
Since $P_{\boldsymbol{\alpha}}$ is compact and $\mathbf{t} + L$ is discrete, the set $R_{\boldsymbol{\alpha}}^P$ is finite. For any $\beta = \sum_{i=1}^n y_i \alpha_i \in \widehat{\sigma}_{\boldsymbol{\alpha}} \cap (\mathbf{t} + L)$, by definition of the upper closure there exists $\delta > 0$ such that for all $\delta' \in (0, \delta)$, the perturbed point $\beta - \delta'(0, \dots, 0, 1) = \sum_{i=1}^n y'_i \alpha_i$ lies in the interior of $\sigma_{\boldsymbol{\alpha}}$ with $y'_i \notin \mathbb{Z}$ for all $i$. Setting $n_i := \lfloor y'_i \rfloor \in \mathbb{N}$ (which is independent of a sufficiently small $\delta'$) and $x_i := y_i - n_i \in [0,1]$, we have $\sum_{i=1}^n x_i \alpha_i - \delta'(0, \dots, 0, 1) = \sum_{i=1}^n \{y'_i\} \alpha_i \in P_{\boldsymbol{\alpha}}$, so $\mathbf{x} = (x_1, \dots, x_n) \in R_{\boldsymbol{\alpha}}^P$ uniquely. Summing over $\mathbf{x} \in R_{\boldsymbol{\alpha}}^P$ and $\mathbf{n} \in \mathbb{N}^n$ gives the result.
\end{proof}

\begin{definition}\label{def: generating}
Let $\sigma_{\boldsymbol{\alpha}}$ be an $n$-dimensional simplicial cone generated by $\boldsymbol{\alpha} = (\alpha_1, \dots, \alpha_n) \in L_{>0}^n$ and let $\mathbf{t} \in V/L$ be a torsion point. We define the \textit{cone generating function} $\mathcal{G}_{\sigma_{\boldsymbol{\alpha}}}(\mathbf{t}, \mathbf{z})$ as the meromorphic function on $\mathbb{C}^n$ given by
\[
    \mathcal{G}_{\sigma_{\boldsymbol{\alpha}}}(\mathbf{t}, \mathbf{z}) := \frac{\sum_{\alpha \in \widehat{P}_{\boldsymbol{\alpha}} \cap (\mathbf{t} + L)} e^{\langle \alpha, \mathbf{z} \rangle}}{\bigl(1 - e^{\langle \alpha_1, \mathbf{z} \rangle}\bigr) \cdots \bigl(1 - e^{\langle \alpha_n, \mathbf{z} \rangle}\bigr)}.
\]
Expanding the geometric series in the region $\Re\langle \alpha_i, \mathbf{z} \rangle < 0$ shows that $\mathcal{G}_{\sigma_{\boldsymbol{\alpha}}}(\mathbf{t}, \mathbf{z}) = \sum_{\alpha \in \widehat{\sigma}_{\boldsymbol{\alpha}} \cap (\mathbf{t} + L)} e^{\langle \alpha, \mathbf{z} \rangle}$, so $\mathcal{G}_{\sigma_{\boldsymbol{\alpha}}}(\mathbf{t}, \mathbf{z})$ depends only on the cone $\sigma_{\boldsymbol{\alpha}}$ and not on the choice of generator $\boldsymbol{\alpha}$.
\end{definition}

\begin{proposition}\label{prop: partial zeta decomp}
Let $\mathbf{t} \in V/L$ be a nonzero torsion point, let $\Gamma = \mathcal{O}_{\mathbf{t}}^{\times,+}$, and let $\Phi$ be a Shintani decomposition. Then
\begin{equation}\label{eq: Shintani}
    \zeta(\mathbf{t}, L, s) = \sum_{\sigma \in \Phi_n / \Gamma} \zeta_\sigma(\mathbf{t}, (s, \dots, s)).
\end{equation}
In particular, for $L = J(\mathfrak{f}\mathfrak{b}^{-1})$, we have $N(\mathfrak{b})^s \zeta(\mathfrak{b}, \mathfrak{f}, s) = \sum_{\sigma \in \Phi_n / \Gamma} \zeta_\sigma(1, (s, \dots, s))$.
\end{proposition}

\begin{proof}
By \eqref{eq: upper closure}, if $C \subset \Phi_n$ is a set of representatives for $\Phi_n/\Gamma$, then $\coprod_{\sigma \in C} \widehat{\sigma}$ is a fundamental domain for the action of $\Gamma$ on $V_{>0}$. The identity then follows from \eqref{eq lattice zeta} and \eqref{eq: diagonal}.
\end{proof}

\subsection{Shintani cocycles}

\paragraph{Lattice and simplicial resolutions.}
For each integer $q \geq 0$, let
\[
    C_q(L_{>0}) := \bigoplus^{\mathrm{alt}}_{\boldsymbol{\alpha} \in L_{>0}^{q+1}} \mathbb{Z}\boldsymbol{\alpha}
\]
be the quotient of $\bigoplus_{\boldsymbol{\alpha} \in L_{>0}^{q+1}} \mathbb{Z}\boldsymbol{\alpha}$ by the submodule generated by $\{\rho(\boldsymbol{\alpha}) - \sgn(\rho)\boldsymbol{\alpha} \mid \boldsymbol{\alpha} \in L_{>0}^{q+1},\, \rho \in \mathfrak{S}_{q+1}\}$ and all tuples $\boldsymbol{\alpha} = (\alpha_0, \dots, \alpha_q)$ with $\alpha_i = \alpha_j$ for some $i \neq j$. Write $\langle \boldsymbol{\alpha} \rangle$ for the class of $\boldsymbol{\alpha}$ in $C_q(L_{>0})$. Equipped with the diagonal action of $\Gamma$ and the standard simplicial differential
\[
    d_q\bigl(\langle \alpha_0, \dots, \alpha_q \rangle\bigr) := \sum_{j=0}^q (-1)^j \langle \alpha_0, \dots, \widehat{\alpha}_j, \dots, \alpha_q \rangle
\]
(augmented by $d_0(\langle \alpha \rangle) = 1$), $C_\bullet(L_{>0})$ is a free resolution of $\mathbb{Z}$ as a $\mathbb{Z}[\Gamma]$-module.

\begin{remark}
Algebraically, the augmented complex $C_\bullet(L_{>0}) \xrightarrow{d_0} \mathbb{Z} \to 0$ can be identified with the (shifted) Koszul complex associated to the augmentation map $\varepsilon: M := \bigoplus_{\alpha \in L_{>0}} \mathbb{Z}e_\alpha \to \mathbb{Z}$ given by $\varepsilon(e_\alpha) = 1$. Indeed, identifying $\langle \alpha_0, \dots, \alpha_q \rangle$ with the exterior product $e_{\alpha_0} \wedge \cdots \wedge e_{\alpha_q}$ gives $C_q(L_{>0}) \cong \bigwedge^{q+1}_{\mathbb{Z}} M$, under which the simplicial differential $d_q$ coincides with the Koszul contraction $\iota_\varepsilon(e_{\alpha_0} \wedge \cdots \wedge e_{\alpha_q}) = \sum_{j=0}^q (-1)^j \varepsilon(e_{\alpha_j})\, e_{\alpha_0} \wedge \cdots \wedge \widehat{e_{\alpha_j}} \wedge \cdots \wedge e_{\alpha_q}$. Since $\varepsilon$ is surjective (any fixed $\beta \in L_{>0}$ defines a splitting $1 \mapsto e_\beta$ and a $\mathbb{Z}$-linear contracting homotopy $e_{\alpha_0} \wedge \cdots \wedge e_{\alpha_q} \mapsto e_\beta \wedge e_{\alpha_0} \wedge \cdots \wedge e_{\alpha_q}$), the Koszul complex is exact, and because $\Gamma$ acts freely on the basis of $M$, each $C_q(L_{>0})$ is a free $\mathbb{Z}[\Gamma]$-module.
\end{remark}

Given a Shintani decomposition $\Phi$, each $\sigma \in \Phi_{q+1}$ has a primitive generator $\boldsymbol{\alpha} = (\alpha_0, \dots, \alpha_q) \in L_{>0}^{q+1}$, unique up to permutation. When $q = n-1$, we set $\langle \sigma \rangle := \sgn(\boldsymbol{\alpha})\langle \boldsymbol{\alpha} \rangle \in C_{n-1}(L_{>0})$, where $\sgn(\boldsymbol{\alpha}) := \sgn\det(A_{\boldsymbol{\alpha}})$ (and similarly for oriented faces in degree $q$).

\begin{lemma}\label{Lemma: quasi iso}
Let $C_q(\Phi)$ be the $\mathbb{Z}[\Gamma]$-submodule of $C_q(L_{>0})$ generated by $\langle \sigma \rangle$ for all $\sigma \in \Phi_{q+1}$. Then $C_\bullet(\Phi)$ is a subcomplex of $C_\bullet(L_{>0})$ that forms a free $\mathbb{Z}[\Gamma]$-resolution of $\mathbb{Z}$. In particular, the inclusion
\[
    C_\bullet(\Phi) \xhookrightarrow{\sim} C_\bullet(L_{>0})
\]
is a $\Gamma$-equivariant quasi-isomorphism.
\end{lemma}

\begin{proof}
Under the radial projection $p_1: V_{>0} \to V^1$, the set $\Phi^1 := \{\sigma \cap V^1 \mid \sigma \in \Phi \setminus \{\{0\}\}\}$ defines a $\Gamma$-equivariant simplicial triangulation of the contractible manifold $V^1 \cong \mathbb{R}^{n-1}$, whose simplicial chain complex is identified with $C_\bullet(\Phi)$. Since $V^1$ is contractible, $C_\bullet(\Phi)$ is an exact free $\mathbb{Z}[\Gamma]$-resolution of $\mathbb{Z}$, and any two free resolutions are quasi-isomorphic.
\end{proof}

\paragraph{Kronecker duality on the unit torus.}
By Dirichlet's Unit Theorem, the quotient $\mathcal{S}^1 := V^1/\Gamma \cong \mathbb{R}^{n-1}/\mathbb{Z}^{n-1}$ is a compact oriented $(n-1)$-dimensional real torus. Passing to $\Gamma$-coinvariants $C_q(\Phi/\Gamma) := C_q(\Phi)_\Gamma$ yields the simplicial chain complex of the triangulation of $\mathcal{S}^1$ induced by $\Phi^1$. Thus,
\[
    H_m\bigl(C_\bullet(\Phi/\Gamma)\bigr) \cong H_m(\mathcal{S}^1, \mathbb{Z}), \qquad H^m\bigl(\Hom_{\mathbb{Z}}(C_\bullet(\Phi/\Gamma), \mathbb{Z})\bigr) \cong H^m(\mathcal{S}^1, \mathbb{Z}).
\]
The fundamental class of $\mathcal{S}^1$ in top degree $H_{n-1}(C_\bullet(\Phi/\Gamma)) \cong \mathbb{Z}$ is represented by the cycle
\begin{equation}\label{eq: fundamental class}
    \sum_{\overline{\sigma} \in \Phi_n/\Gamma} \langle \overline{\sigma} \rangle \in C_{n-1}(\Phi/\Gamma).
\end{equation}
For any field $\mathbb{K}$, the Universal Coefficient Theorem and Kronecker duality induce the canonical evaluation isomorphism
\begin{equation}\label{eq: isom}
\begin{tikzcd}[row sep=tiny]
    H^{n-1}\bigl(\Hom_{\mathbb{Z}}(C_\bullet(\Phi/\Gamma), \mathbb{K})\bigr) \arrow[rr, "\cong"] && \mathbb{K} \\
    {[\varphi]} \arrow[rr, maps to] && \sum_{\overline{\sigma} \in \Phi_n/\Gamma} \varphi\bigl(\langle \overline{\sigma} \rangle\bigr).
\end{tikzcd}
\end{equation}

\paragraph{The Shintani cocycle class.}
By \cite[Proposition 6.2]{Yam10}, for any $\boldsymbol{\alpha} = (\alpha_0, \dots, \alpha_n) \in L_{>0}^{n+1}$, the characteristic functions of the upper closures satisfy the cocycle identity on $V_{>0}$:
\begin{equation}\label{eq: cocycle}
    \sum_{j=0}^n (-1)^j \sgn(\alpha_0, \dots, \widehat{\alpha}_j, \dots, \alpha_n) \mathbf{1}_{\widehat{\sigma}_{(\alpha_0, \dots, \widehat{\alpha}_j, \dots, \alpha_n)}} \equiv 0.
\end{equation}
Multiplying by $e^{\langle \alpha, \mathbf{z} \rangle}$ and summing over $\alpha \in \mathbf{t} + L$ yields the cocycle relation for the cone generating functions:
\[
    \sum_{j=0}^n (-1)^j \sgn(\alpha_0, \dots, \widehat{\alpha}_j, \dots, \alpha_n) \mathcal{G}_{\sigma_{(\alpha_0, \dots, \widehat{\alpha}_j, \dots, \alpha_n)}}(\mathbf{t}, \mathbf{z}) = 0.
\]

\begin{definition}\label{def shintani class}
Let $\mathbb{K}$ be a subfield of $\mathbb{C}$ containing $F^{\mathrm{Gal}}$, let $\widetilde{R} = \widetilde{R}_{\mathbb{K}}$ be the field of cone-adapted Laurent series introduced in Section \ref{section Field of Coefficients} above, and let $\rho_\Gamma = \rho_N: \widetilde{R} \twoheadrightarrow \widetilde{R}_\Gamma \cong \mathbb{K}(\!(w)\!)$ be the diagonal norm projection. For a fixed nonzero torsion point $\mathbf{t} \in V/L$ stabilized by $\Gamma$, we define the \textit{Shintani class}
\[
    [\mathcal{G}] \in H^{n-1}\bigl(\Hom_{\mathbb{Z}}(C_\bullet(\Phi/\Gamma), \widetilde{R}_\Gamma)\bigr)
\]
as the cohomology class of the $(n-1)$-cocycle $\langle \boldsymbol{\alpha} \rangle \mapsto \rho_\Gamma\bigl(\sgn(\boldsymbol{\alpha})\mathcal{G}_{\sigma_{\boldsymbol{\alpha}}}(\mathbf{t}, \mathbf{z})\bigr)$ (or equivalently $\langle \overline{\sigma}_{\boldsymbol{\alpha}} \rangle \mapsto \rho_\Gamma\bigl(\mathcal{G}_{\sigma_{\boldsymbol{\alpha}}}(\mathbf{t}, \mathbf{z})\bigr)$).
\end{definition}

\section{Explicit Construction of the Green Current and Main Theorem}\label{sec:explicit_comparison}

In this section we combine the theories developed in the two previous sections and construct a suitable Green current which computes the polylogarithm class. The basic idea is to decompose the geometric spaces of section \ref{sec:polylog} into cones and show that the polylogarithm class is compatible with such a decomposition and specializes to the corresponding Shintani class.

\subsection{Hurwitz's integral formula as a Mellin transform on $X$}\label{sec:hurwitz_mellin}

Remember that through $\psi$, the space $X$ inherits the $\Gamma$-action from $V^1$, and any Shintani decomposition $\Phi$ of $V_{>0} \cup \{0\}$ induces compatible $\Gamma$-equivariant simplicial decompositions $\Phi^1$ of $V^1$ and $\Phi^X := \{X_\sigma := \sigma \cap X \mid \sigma \in \Phi \setminus \{\{0\}\}\}$ of $X$.

For any tuple $\boldsymbol{\alpha} = (\alpha_1, \dots, \alpha_n) \in L_{>0}^n$ generating an $n$-dimensional simplicial cone $\sigma_{\boldsymbol{\alpha}} \subset V_{>0} \cup \{0\}$ with matrix $A_{\boldsymbol{\alpha}} = J(\boldsymbol{\alpha})$, consider the rational function of Sczech \cite{Scz93}:
\[
    f(\boldsymbol{\alpha})(\mathbf{z}) := \frac{\det(A_{\boldsymbol{\alpha}})}{\prod_{j=1}^n \langle \mathbf{z}, \alpha_j \rangle} \in \widetilde{R}_{F^{\mathrm{Gal}}},
\]
which satisfies the cocycle relation $\sum_{i=0}^n (-1)^i f(\alpha_0, \dots, \widehat{\alpha}_i, \dots, \alpha_n) = 0$ (\cite[Lemma 1]{Scz93}).

In the Proposition below we prove Hurwitz's integral representation for $f(\boldsymbol{\alpha})(\mathbf{z})$ directly by interpreting it as a Mellin transform on $V_{>0}$ factored through the radial projection to $X = \mathbb{S}^{n-1} \cap V_{>0}$.  

\begin{proposition}[Hurwitz \cite{Hur22}]\label{hurwitz}
Let $X_{\boldsymbol{\alpha}} := \sigma_{\boldsymbol{\alpha}} \cap X$ be the spherical simplex defined by an oriented simplicial cone $\sigma_{\boldsymbol{\alpha}}$ ($\det(A_{\boldsymbol{\alpha}}) > 0$), and let $\mathbf{z} \in V_{\mathbb{C}}$ be such that $\langle \mathbf{z}, \alpha_j \rangle \neq 0$ for all $j = 1, \dots, n$. Define the $(n-1)$-form on $X$ (and $V$) by
\[
    \omega_y := \sum_{i=1}^n (-1)^{i-1} y_i\, dy_1 \wedge \cdots \wedge \widehat{dy_i} \wedge \cdots \wedge dy_n \in \Omega^{n-1}(X).
\]
Then
\begin{equation}\label{eq:hurwitz_formula}
    (n-1)! \int_{X_{\boldsymbol{\alpha}}} \frac{\omega_y}{\langle \mathbf{z}, y \rangle^n} = f(\boldsymbol{\alpha})(\mathbf{z}).
\end{equation}
\end{proposition}

\begin{proof}
By analytic continuation, it suffices to prove \eqref{eq:hurwitz_formula} for $\mathbf{z} \in V_{\mathbb{C}}$ with $\Re(z_i) > 0$ for all $i = 1, \dots, n$, so that $\Re\langle \mathbf{z}, v \rangle > 0$ for all $v \in V_{>0}$. Consider the Laplace integral
\[
    I(\boldsymbol{\alpha}, \mathbf{z}) := \int_{\sigma_{\boldsymbol{\alpha}}} e^{-\langle \mathbf{z}, v \rangle}\, dv_1 \wedge \cdots \wedge dv_n.
\]
On the one hand, parameterizing $\sigma_{\boldsymbol{\alpha}}$ by $v = \sum_{j=1}^n t_j \alpha_j$ with $t_j \in \mathbb{R}_{\geq 0}$ gives $dv_1 \wedge \cdots \wedge dv_n = \det(A_{\boldsymbol{\alpha}})\, dt_1 \wedge \cdots \wedge dt_n$, so
\[
    I(\boldsymbol{\alpha}, \mathbf{z}) = \det(A_{\boldsymbol{\alpha}}) \prod_{j=1}^n \int_0^\infty e^{-t_j \langle \mathbf{z}, \alpha_j \rangle}\, dt_j = \frac{\det(A_{\boldsymbol{\alpha}})}{\prod_{j=1}^n \langle \mathbf{z}, \alpha_j \rangle} = f(\boldsymbol{\alpha})(\mathbf{z}).
\]
On the other hand, under the radial diffeomorphism $\mathbb{R}_{>0} \times X \xrightarrow{\sim} V_{>0}$, $(r, y) \mapsto v = ry$, we have $dv_i = y_i\,dr + r\,dy_i$ and hence $dv_1 \wedge \cdots \wedge dv_n = r^{n-1} dr \wedge \omega_y$, while $\sigma_{\boldsymbol{\alpha}} \setminus \{0\} \cong \mathbb{R}_{>0} \times X_{\boldsymbol{\alpha}}$. Integrating first over $r \in (0, \infty)$ (which is the Mellin transform of $e^{-r\langle \mathbf{z}, y \rangle}$ at $s = n$) yields
\[
    I(\boldsymbol{\alpha}, \mathbf{z}) = \int_{X_{\boldsymbol{\alpha}}} \left(\int_0^\infty e^{-r\langle \mathbf{z}, y \rangle} r^{n-1}\,dr\right) \omega_y = (n-1)! \int_{X_{\boldsymbol{\alpha}}} \frac{\omega_y}{\langle \mathbf{z}, y \rangle^n}.
\]
Equating the two expressions for $I(\boldsymbol{\alpha}, \mathbf{z})$ gives \eqref{eq:hurwitz_formula}.

\end{proof}

\begin{remark}
    Hurwitz proves this in \cite{Hur22} (see also \cite{Scz93} \S 2.1) by considering the integral over the unique simplex of $\mathbb{P}^{n-1}(\mathbb{R})$ defined by $\boldsymbol{\alpha}$ where $\langle \mathbf{z}, y \rangle$ does not vanish. When $\boldsymbol{\alpha}$ consists of totally positive generators, as in our case, this simplex corresponds to $X_{\boldsymbol{\alpha}}$ under the equivalence between $\mathbb{P}^{n-1}(\mathbb{R})$ and the $n-1$ sphere modulo the antipodal map.
\end{remark}

\begin{corollary}\label{corollary to hurwitz}
Let $\psi: V^1 \xrightarrow{\sim} X$, $\psi(q) = \frac{q}{\Vert q \Vert}$, be the radial diffeomorphism, and on $V^1 = \{q \in V_{>0} \mid N(q) = 1\}$ set $\omega_{V^1} := \sum_{i=1}^n (-1)^{i-1} d\log q_1 \wedge \cdots \wedge \widehat{d\log q_i} \wedge \cdots \wedge d\log q_n$. Then the degree-$0$ homogeneous $(n-1)$-form $\frac{\omega_y}{\langle \mathbf{z}, y \rangle^n}$ satisfies $\psi^*\left(\frac{\omega_y}{\langle \mathbf{z}, y \rangle^n}\right) = \frac{\omega_{V^1}}{\langle \mathbf{z}, q \rangle^n}$, and on the entire positive orthant $X = X_{(e_1, \dots, e_n)}$ we have
\begin{equation}\label{eq:hurwitz_global}
    (n-1)! \int_X \frac{\omega_y}{\langle \mathbf{z}, y \rangle^n} = (n-1)! \int_{V^1} \frac{\omega_{V^1}}{\langle \mathbf{z}, q \rangle^n} = f(e_1, \dots, e_n)(\mathbf{z}) = \frac{1}{N(\mathbf{z})}.
\end{equation}
\end{corollary}

\begin{proof}
Since $\frac{\omega_y}{\langle \mathbf{z}, y \rangle^n}$ is invariant under rescaling $y \mapsto \lambda y$ for any smooth positive function $\lambda$ (as $d(\lambda y_i) = \lambda\,dy_i + y_i\,d\lambda$ and $\sum_{i=1}^n (-1)^{i-1} y_i\,dy_1 \wedge \cdots \wedge y_j\,d\lambda \wedge \cdots \wedge dy_n = 0$), its pullback along $\psi(q) = \Vert q \Vert^{-1} q$ is simply the restriction of $\frac{\omega_q}{\langle \mathbf{z}, q \rangle^n}$ to $V^1$. Since $q_1 \cdots q_n = 1$ on $V^1$, we have $\omega_q|_{V^1} = \omega_{V^1}$. Taking $\boldsymbol{\alpha} = (e_1, \dots, e_n)$ in Proposition \ref{hurwitz} gives $X_{\boldsymbol{\alpha}} = X$, $\det(A_{\boldsymbol{\alpha}}) = 1$, and $\prod_{j=1}^n \langle \mathbf{z}, e_j \rangle = z_1 \cdots z_n = N(\mathbf{z})$. Note that while $\boldsymbol{\alpha} = (e_1, \dots, e_n)$ may not be in $L$ and is not strictly positive, the proof of Proposition \ref{hurwitz} still applies.
\end{proof}

\begin{remark}
    This is Theorem 5.6 of \cite{BKL18}. It is also equivalent to \cite[(0.3)]{Col24}. There, in (0.2) Colmez decomposes the totally positive orthant into cones as in Definition \ref{def: Shintani} and takes the Laplace transform of the identification function, which gives (0.3) equating $\frac{1}{N( \boldsymbol{z})}$ and the sum of all the functions $f(\boldsymbol{\alpha})(\boldsymbol{z})$ for all cones. This sum is precisely the integral in Proposition \ref{hurwitz} decomposed into the simplexes defined by each cone.
\end{remark}

\subsection{Poisson summation and cone generating distributions}

Throughout the remainder of this section, we set $\widetilde{R} = \widetilde{R}_{\mathbb{C}} = \mathbb{C}_{\leq_{\mathrm{lex}}}(\!(\mathbf{z})\!)$.

\begin{proposition}\label{prop: equality of distributions}
Let $\sigma_{\boldsymbol{\alpha}}$ be an $n$-dimensional simplicial cone generated by $\boldsymbol{\alpha} = (\alpha_1, \dots, \alpha_n) \in L_{>0}^n$. Then the cone generating function $\mathcal{G}_{\sigma_{\boldsymbol{\alpha}}}(\mathbf{t}, \mathbf{z})$ of Definition \ref{def: generating} defines an $\widetilde{R}$-valued distribution on $T$ (smooth away from the subtori orthogonal to $\alpha_1, \dots, \alpha_n$) satisfying the Fourier identity
\begin{equation}\label{eq:poisson_cone}
    \vol(L)^{-1} \sum_{\mu \in L^*} e^{2\pi i \langle \mathbf{t}, \mu \rangle} f(\boldsymbol{\alpha})(2\pi i \mu - \mathbf{z}) = \sgn(\boldsymbol{\alpha}) \mathcal{G}_{\sigma_{\boldsymbol{\alpha}}}(\mathbf{t}, \mathbf{z}).
\end{equation}
\end{proposition}

\begin{proof}
Since $\sigma_{\boldsymbol{\alpha}} \subset V_{>0} \cup \{0\}$ belongs to $\mathscr{C}_{\mathrm{lex}}$, each factor $1 - e^{\langle \alpha_j, \mathbf{z} \rangle} = -\langle \alpha_j, \mathbf{z} \rangle(1 + \frac{1}{2}\langle \alpha_j, \mathbf{z} \rangle + \cdots)$ is nonzero in $\mathbb{C}_{\leq_{\mathrm{lex}}}\llbracket \mathbf{z} \rrbracket$ and hence invertible in the field $\widetilde{R}_{\mathbb{C}}$, so $\mathcal{G}_{\sigma_{\boldsymbol{\alpha}}}(\mathbf{t}, \mathbf{z}) \in \widetilde{R}_{\mathbb{C}}$.

To prove \eqref{eq:poisson_cone}, we follow \cite[Lemma 18]{BV00}. Lift to the universal cover $V$ of $T$, and let $V_{\sigma_{\boldsymbol{\alpha}}} \subset V$ be the complement of all $L$-translates of the walls of $\sigma_{\boldsymbol{\alpha}}$. Let $L_{\boldsymbol{\alpha}} := \bigoplus_{j=1}^n \mathbb{Z}\alpha_j \subset L$ be the sublattice generated by $\boldsymbol{\alpha}$. For each $\mathbf{t} \in V_{\sigma_{\boldsymbol{\alpha}}}$, the finite quotient $L/L_{\boldsymbol{\alpha}}$ of cardinality $|L/L_{\boldsymbol{\alpha}}| = |\det(A_{\boldsymbol{\alpha}})|/\vol(L)$ is represented by $\mathscr{R}(\mathbf{t}, \boldsymbol{\alpha}) := (\mathbf{t} - \widehat{P}_{\boldsymbol{\alpha}}) \cap L$. Setting
\[
    F_{\boldsymbol{\alpha}}(\mathbf{t}, \mathbf{z}) := |L/L_{\boldsymbol{\alpha}}|^{-1} \sum_{x \in \mathscr{R}(\mathbf{t}, \boldsymbol{\alpha})} \frac{e^{\langle x, \mathbf{z} \rangle}}{\prod_{j=1}^n (1 - e^{-\langle \alpha_j, \mathbf{z} \rangle})},
\]
\cite[Lemma 18]{BV00} (via Poisson summation on $L$) yields the distributional identity on $V_{\sigma_{\boldsymbol{\alpha}}}$:
\[
    F_{\boldsymbol{\alpha}}(\mathbf{t}, -\mathbf{z}) = \sum_{\mu \in L^*} e^{\langle \mathbf{t}, 2\pi i \mu - \mathbf{z} \rangle} f(\boldsymbol{\alpha})(2\pi i \mu - \mathbf{z}).
\]
Using the bijection $(\mathbf{t} - \widehat{P}_{\boldsymbol{\alpha}}) \cap L \xrightarrow{\sim} \widehat{P}_{\boldsymbol{\alpha}} \cap (\mathbf{t} + L)$, $x \mapsto \mathbf{t} - x$, together with $|L/L_{\boldsymbol{\alpha}}| = \sgn(\boldsymbol{\alpha})\det(A_{\boldsymbol{\alpha}})/\vol(L)$, and multiplying both sides by $e^{\langle \mathbf{t}, \mathbf{z} \rangle}$ gives \eqref{eq:poisson_cone}.
\end{proof}

\begin{remark}\label{rmk cone avoids t}
For a fixed nonzero torsion point $\mathbf{t} \in \overset{\circ}{T}$, any Shintani decomposition $\Phi$ can be slightly deformed into a Shintani decomposition $\Phi'$ so that $\mathbf{t}$ does not lie on the $L$-translates of the walls of any cone in $\Phi'$ (compare \cite[Lemma 5.3]{BHY19}). Thus we may always assume without loss of generality that $\mathbf{t} \in V_{\sigma_{\boldsymbol{\alpha}}} / L$ for all $\sigma_{\boldsymbol{\alpha}} \in \Phi_n$, so that \eqref{eq:poisson_cone} holds pointwise at $\mathbf{t}$.
\end{remark}

\subsection{Explicit construction of the Green current on $X \times T$}\label{sec:direct_green}

Working over $\widetilde{R}_{\mathbb{C}} = \mathbb{C}_{\leq_{\mathrm{lex}}}(\!(\mathbf{z})\!)$, our goal is to construct a Green current for the topological polylogarithm directly on $X \times T$.

\paragraph{Smooth pro-bundles and canonical trivializations.}
Setting $R^{(k)} := R_{\mathbb{C}}/I^{k+1}$ and $\mathcal{L}\mathrm{og}^{(k)} := \mathcal{L}\mathrm{og} \otimes_{\underline{R}_{\mathbb{C}}} \underline{R}^{(k)}$, we define the associated smooth pro-bundles on $T$ by
\[
    \underline{R}_{\mathbb{C}}^\infty := \varprojlim_k (\mathscr{C}^\infty \otimes_{\mathbb{C}} \underline{R}^{(k)}) \cong \mathscr{C}^\infty\llbracket \mathbf{z} \rrbracket, \qquad \mathcal{L}\mathrm{og}^\infty := \varprojlim_k (\mathscr{C}^\infty \otimes_{\mathbb{C}} \mathcal{L}\mathrm{og}^{(k)}) \cong \underline{R}_{\mathbb{C}}^\infty \otimes_{\underline{R}_{\mathbb{C}}} \mathcal{L}\mathrm{og},
\]
and write $\Omega^i \widehat{\otimes} R_{\mathbb{C}} := \varprojlim_k (\Omega^i \otimes_{\mathbb{C}} R^{(k)}) \cong \Omega^i\llbracket \mathbf{z} \rrbracket$ and $\widehat{\Omega}^i \widehat{\otimes} R_{\mathbb{C}} := \varprojlim_k (\widehat{\Omega}^i \otimes_{\mathbb{C}} R^{(k)})$ for the sheaves of $R_{\mathbb{C}}$-valued differential forms and currents. Passing to the flat $R_{\mathbb{C}}$-algebra $\widetilde{R}_{\mathbb{C}} = \mathbb{C}_{\leq_{\mathrm{lex}}}(\!(\mathbf{z})\!)$, we set
\[
    \underline{\widetilde{R}}_{\mathbb{C}}^\infty := \underline{R}_{\mathbb{C}}^\infty \otimes_{\underline{R}_{\mathbb{C}}} \underline{\widetilde{R}}_{\mathbb{C}}, \qquad \widetilde{\mathcal{L}\mathrm{og}}^\infty := \mathcal{L}\mathrm{og}^\infty \otimes_{\underline{R}_{\mathbb{C}}} \underline{\widetilde{R}}_{\mathbb{C}} \cong \underline{\widetilde{R}}_{\mathbb{C}}^\infty \otimes_{\underline{\widetilde{R}}_{\mathbb{C}}} \widetilde{\mathcal{L}\mathrm{og}},
\]
and analogously $\Omega^i \widehat{\otimes} \widetilde{R}_{\mathbb{C}} := (\Omega^i \widehat{\otimes} R_{\mathbb{C}}) \otimes_{R_{\mathbb{C}}} \widetilde{R}_{\mathbb{C}}$ and $\widehat{\Omega}^i \widehat{\otimes} \widetilde{R}_{\mathbb{C}} := (\widehat{\Omega}^i \widehat{\otimes} R_{\mathbb{C}}) \otimes_{R_{\mathbb{C}}} \widetilde{R}_{\mathbb{C}}$.

\begin{lemma}\label{lem:continuous_trivialization}
The smooth map $V \to V \times (\underline{R}_{\mathbb{C}}^\infty)^\times$, $v \mapsto (v, e^{-\langle v, \mathbf{z} \rangle})$, descends to a global smooth section $\varrho_{\mathrm{cont}}$ of $\mathcal{L}\mathrm{og}^\infty$ on $T = V/L$ inducing an isomorphism of smooth pro-bundles with flat connection
\[
    \varrho_{\mathrm{cont}}: (\underline{R}_{\mathbb{C}}^\infty, d - \kappa \wedge \cdot) \xrightarrow{\sim} (\mathcal{L}\mathrm{og}^\infty, \nabla),
\]
where $\kappa := d_V\langle v, \mathbf{z} \rangle = \sum_{j=1}^n z_j\, dv_j \in \Omega^1(T) \widehat{\otimes} R_{\mathbb{C}}$. Moreover, base change along $R_{\mathbb{C}} \hookrightarrow \widetilde{R}_{\mathbb{C}}$ induces $\varrho_{\mathrm{cont}}: (\underline{\widetilde{R}}_{\mathbb{C}}^\infty, d - \kappa \wedge \cdot) \xrightarrow{\sim} (\widetilde{\mathcal{L}\mathrm{og}}^\infty, \nabla)$.
\end{lemma}

\begin{proof}
For all $v \in V$ and $l \in L$, we have $(v+l, e^{-\langle v+l, \mathbf{z} \rangle}) = (v+l, e^{-\langle l, \mathbf{z} \rangle} e^{-\langle v, \mathbf{z} \rangle}) = (v, e^{-\langle v, \mathbf{z} \rangle}) \cdot l$, so $v \mapsto (v, e^{-\langle v, \mathbf{z} \rangle})$ descends to a global smooth section $\varrho_{\mathrm{cont}}$ of $\mathcal{L}\mathrm{og}^\infty$ on $T$. Since $e^{-\langle v, \mathbf{z} \rangle}$ is a unit in $\underline{R}_{\mathbb{C}}^\infty$ with inverse $e^{\langle v, \mathbf{z} \rangle}$, multiplication by $\varrho_{\mathrm{cont}}$ is an isomorphism. Finally, since $\nabla = d_V \otimes \id$ on the universal cover $V$, we have
\[
    \nabla(\varrho_{\mathrm{cont}}) = d_V\bigl(e^{-\langle v, \mathbf{z} \rangle}\bigr) = -\bigl(d_V\langle v, \mathbf{z} \rangle\bigr) e^{-\langle v, \mathbf{z} \rangle} = -\kappa\, \varrho_{\mathrm{cont}},
\]
so for any section $f$ of $\underline{R}_{\mathbb{C}}^\infty$, $\nabla(f\,\varrho_{\mathrm{cont}}) = (df - \kappa \wedge f)\,\varrho_{\mathrm{cont}}$.
\end{proof}

\begin{remark}
    On the subgroup $T^{(A)} \subset T_{\mathrm{tors}}$ of torsion points of order invertible in $A$, $\mathcal{L}\mathrm{og}$ admits a canonical multiplicative trivialization $\varrho_{\mathrm{can}}$, which agrees with $\varrho_{\mathrm{cont}}|_{T_{\mathrm{tors}}}$ when $A = \mathbb{C}$ (see \cite[Proposition 3.24 and Lemma 4.2]{BKL18}).
\end{remark}

\paragraph{Invertibility and uniform lower bound of linear forms on $X$.}
For every $y = (y_1, \dots, y_n) \in X \subset V_{>0}$ and every $\mu \in L^*$ (including $\mu = 0$), consider the linear form
\[
    \langle 2\pi i \mu - \mathbf{z}, y \rangle = 2\pi i \langle \mu, y \rangle - \sum_{j=1}^n y_j z_j \in \mathbb{C}\llbracket \mathbf{z} \rrbracket.
\]
Since $y_j > 0$ for all $j = 1, \dots, n$, the homogeneous linear part $\langle \mathbf{z}, y \rangle = \sum_{j=1}^n y_j z_j$ is a nonzero linear form with positive coefficients. Thus:
\begin{itemize}
    \item If $\langle \mu, y \rangle \neq 0$, then $\langle 2\pi i \mu - \mathbf{z}, y \rangle$ has nonzero constant term $2\pi i \langle \mu, y \rangle \in \mathbb{C}^\times$ and is already invertible in $R_{\mathbb{C}} = \mathbb{C}\llbracket \mathbf{z} \rrbracket$:
    \[
        \frac{1}{\langle 2\pi i \mu - \mathbf{z}, y \rangle^{a+1}} = \sum_{k=0}^\infty \binom{k+a}{a} \frac{\langle \mathbf{z}, y \rangle^k}{(2\pi i \langle \mu, y \rangle)^{k+a+1}} \in R_{\mathbb{C}}.
    \]
    \item If $\langle \mu, y \rangle = 0$ (in particular, for $\mu = 0$), then $\langle 2\pi i \mu - \mathbf{z}, y \rangle = -\langle \mathbf{z}, y \rangle$, which is invertible in $\widetilde{R}_{\mathbb{C}}$.
    \item Analytically, for $\mathbf{z} \in \mathbb{C}^n$ with $\Re(z_j) \leq -c < 0$ ($j = 1, \dots, n$), since $2\pi i \langle \mu, y \rangle$ is purely imaginary while $-\langle \Re(\mathbf{z}), y \rangle \geq c \sum_{j=1}^n y_j \geq c > 0$, we have the uniform lower bound
    \begin{equation}\label{eq:uniform_bound}
        \bigl|\langle 2\pi i \mu - \mathbf{z}, y \rangle\bigr| = \sqrt{\bigl(2\pi \langle \mu, y \rangle - \langle \Im(\mathbf{z}), y \rangle\bigr)^2 + \langle \Re(\mathbf{z}), y \rangle^2} \geq \bigl|\langle \Re(\mathbf{z}), y \rangle\bigr| \geq c > 0
    \end{equation}
    for all $y \in X$ and all $\mu \in L^*$, and the Laplace integral
    \[
        \int_0^\infty e^{-u\langle 2\pi i \mu - \mathbf{z}, y \rangle} u^a\, du = \frac{a!}{\langle 2\pi i \mu - \mathbf{z}, y \rangle^{a+1}}
    \]
    converges absolutely and uniformly on $X \times L^*$ for every $a \geq 0$.
\end{itemize}

\paragraph{Theta-form construction and fibre integration.}
For the explicit construction of our Green current we follow a method due to Prader, which uses the Mellin transform of a certain theta form defined below. Remember from Lemma \ref{lem:continuous_trivialization} that under the continuous multiplicative trivialization $\varrho_{\mathrm{cont}}: \underline{\widetilde{R}}_{\mathbb{C}}^\infty \xrightarrow{\sim} \widetilde{\mathcal{L}\mathrm{og}}^\infty$ induced by $v \mapsto (v, e^{-\langle v, \mathbf{z} \rangle})$, the flat connection $\nabla$ on $X \times T$ is given by
\[
    \nabla = d - \kappa \wedge \cdot, \qquad \kappa := \sum_{j=1}^n z_j\, dv_j = d_V\langle v, \mathbf{z} \rangle, \qquad \nabla \eta = e^{\langle v, \mathbf{z} \rangle} d\left(e^{-\langle v, \mathbf{z} \rangle} \eta\right).
\]
For each $\mu \in L^*$, define the closed smooth $n$-form on $V = \mathbb{R}^n$ (with coordinates $\zeta = (\zeta_1, \dots, \zeta_n)$) by
\[
    \alpha_\mu(\zeta) := e^{\langle \zeta, 2\pi i \mu - \mathbf{z} \rangle} d\zeta_1 \wedge \cdots \wedge d\zeta_n \in \Omega^n(V) \widehat{\otimes} \widetilde{R}_{\mathbb{C}}, \qquad d\alpha_\mu = 0
\]
(which is closed since $\dim_{\mathbb{R}} V = n$). Along the smooth map $\Phi: X \times V \times \mathbb{R}_{>0} \to V$, $\Phi(y, v, u) := v - uy$, define the twisted pullback $n$-form
\[
    \theta_\mu(y, v, u) := e^{\langle v, \mathbf{z} \rangle} \Phi^* \alpha_\mu = e^{2\pi i \langle v, \mu \rangle} e^{-u\langle 2\pi i \mu - \mathbf{z}, y \rangle} \bigwedge_{j=1}^n (dv_j - u\, dy_j - y_j\, du).
\]
Since $\theta_\mu$ is $L$-invariant, it descends to $X \times T \times \mathbb{R}_{>0}$, and since $d\alpha_\mu = 0$, it is $\nabla$-closed:
\[
    \nabla \theta_\mu = e^{\langle v, \mathbf{z} \rangle} d(\Phi^* \alpha_\mu) = e^{\langle v, \mathbf{z} \rangle} \Phi^*(d\alpha_\mu) = 0.
\]
Decompose $\theta_\mu = du \wedge \theta_\mu^{(1)} + \theta_\mu^{(2)}$, where $\theta_\mu^{(1)}$ and $\theta_\mu^{(2)}$ have degree $0$ in $du$:
\[
    \theta_\mu^{(1)}(y, v, u) = -e^{2\pi i \langle v, \mu \rangle} e^{-u\langle 2\pi i \mu - \mathbf{z}, y \rangle} \sum_{i=1}^n (-1)^{i-1} y_i \bigwedge_{j \neq i} (dv_j - u\, dy_j),
\]
\[
    \theta_\mu^{(2)}(y, v, u) = e^{2\pi i \langle v, \mu \rangle} e^{-u\langle 2\pi i \mu - \mathbf{z}, y \rangle} \bigwedge_{j=1}^n (dv_j - u\, dy_j).
\]
Then $0 = \nabla \theta_\mu = -du \wedge \nabla_{X \times T} \theta_\mu^{(1)} + du \wedge \partial_u \theta_\mu^{(2)} + \nabla_{X \times T} \theta_\mu^{(2)}$ implies the identity
\[
    \nabla_{X \times T} \theta_\mu^{(1)} = \frac{\partial \theta_\mu^{(2)}}{\partial u}.
\]

\begin{lemma}\label{lem:fibre_integral_green}
For each $\mu \in L^*$, define $\mathscr{G}_\mu(y, v) \in \Omega^{n-1}(X \times T) \widehat{\otimes} \widetilde{R}_{\mathbb{C}}$ by the fibre integral
\[
    \mathscr{G}_\mu(y, v) := -\int_0^\infty \theta_\mu^{(1)}(y, v, u)\, du.
\]
Then:
\begin{enumerate}
    \item $\mathscr{G}_\mu$ satisfies the differential equation
    \[
        \nabla_{X \times T} \mathscr{G}_\mu(y, v) = e^{2\pi i \langle v, \mu \rangle} dv_1 \wedge \cdots \wedge dv_n.
    \]
    \item In the bidegree decomposition $\Omega^{n-1}(X \times T) = \bigoplus_{a=0}^{n-1} \Omega^a(X) \otimes \Omega^{n-1-a}(T)$, we have $\mathscr{G}_\mu = \sum_{a=0}^{n-1} \mathscr{G}_\mu^a$ with
    \[
        \mathscr{G}_\mu^a(y, v) = (-1)^a a! \frac{e^{2\pi i \langle v, \mu \rangle}}{\langle 2\pi i \mu - \mathbf{z}, y \rangle^{a+1}} \sum_{i=1}^n (-1)^{i-1} y_i\, \omega_{i,a}(y, v),
    \]
    where $\bigwedge_{j \neq i} (dv_j - u\, dy_j) = \sum_{a=0}^{n-1} (-u)^a \omega_{i,a}(y, v)$. In particular, the top-$X$-degree component ($a = n-1$) is
    \[
        \mathscr{G}_\mu^{n-1}(y, v) = (-1)^{n-1} (n-1)! \frac{e^{2\pi i \langle v, \mu \rangle}}{\langle 2\pi i \mu - \mathbf{z}, y \rangle^n} \omega_y.
    \]
\end{enumerate}
\end{lemma}

\begin{proof}
(1) Using $\nabla_{X \times T} \theta_\mu^{(1)} = \partial_u \theta_\mu^{(2)}$, we obtain
\[
    \nabla_{X \times T} \mathscr{G}_\mu(y, v) = -\int_0^\infty \frac{\partial \theta_\mu^{(2)}}{\partial u}(y, v, u)\, du = -\left[\theta_\mu^{(2)}(y, v, u)\right]_{u=0}^{u=\infty} = \theta_\mu^{(2)}(y, v, 0) = e^{2\pi i \langle v, \mu \rangle} dv_1 \wedge \cdots \wedge dv_n.
\]
(2) Expanding $\bigwedge_{j \neq i} (dv_j - u\, dy_j) = \sum_{a=0}^{n-1} (-u)^a \omega_{i,a}(y, v)$ and evaluating the Gamma integral $\int_0^\infty e^{-u\langle 2\pi i \mu - \mathbf{z}, y \rangle} u^a\, du = \frac{a!}{\langle 2\pi i \mu - \mathbf{z}, y \rangle^{a+1}}$ gives the formula for $\mathscr{G}_\mu^a(y, v)$. When $a = n-1$, $\omega_{i,n-1} = dy_1 \wedge \cdots \wedge \widehat{dy_i} \wedge \cdots \wedge dy_n$ and $\sum_{i=1}^n (-1)^{i-1} y_i\, \omega_{i,n-1} = \omega_y$.
\end{proof}

\begin{proposition}\label{prop green current}
Let $\vol_T := \vol(L)^{-1} dv_1 \wedge \cdots \wedge dv_n$ be the normalized volume form on $T = V/L$ with $\int_T \vol_T = 1$, and define the $(n-1)$-currents on $X \times T$ by the normalized Fourier series
\[
    \mathscr{G} := \vol(L)^{-1} \sum_{\mu \in L^* \setminus \{0\}} \mathscr{G}_\mu \quad \text{and} \quad \widetilde{\mathscr{G}} := \vol(L)^{-1} \sum_{\mu \in L^*} \mathscr{G}_\mu = \mathscr{G} + \vol(L)^{-1}\mathscr{G}_0.
\]
Then for $\Re(z_j) < 0$, $\mathscr{G}$ and $\widetilde{\mathscr{G}}$ converge unconditionally in the space of currents $\widehat{\Omega}^{n-1}(X \times T)$ (as smooth forms in $y \in X$ valued in distributions on $T$) and satisfy the current differential equations on $X \times T$:
\[
    \nabla \mathscr{G} = \delta_{X \times \{0\}} \vol_T - \vol_T \quad \text{and} \quad \nabla \widetilde{\mathscr{G}} = \delta_{X \times \{0\}} \vol_T.
\]
In particular, $\widetilde{\mathscr{G}}$ represents the topological polylogarithm class $\pol \in H^{n-1}(T \setminus \{0\}, \Gamma; \widetilde{\mathcal{L}\mathrm{og}})$ of Definition \ref{def polylog}. Furthermore, for any constant nonzero torsion section $\mathbf{t}: X \to X \times \overset{\circ}{T}$, we have $\mathbf{t}^*\vol_T = 0$, so both $\mathbf{t}^*\mathscr{G}$ and $\mathbf{t}^*\widetilde{\mathscr{G}}$ are closed on $X$ ($d_X(\mathbf{t}^*\mathscr{G}) = d_X(\mathbf{t}^*\widetilde{\mathscr{G}}) = 0$) and are given by
\begin{equation}\label{eq:t_pullback_G}
    \mathbf{t}^* \mathscr{G} = \frac{(-1)^{n-1} (n-1)!}{\vol(L)} \sum_{\mu \in L^* \setminus \{0\}} \frac{e^{2\pi i \langle \mathbf{t}, \mu \rangle}}{\langle 2\pi i \mu - \mathbf{z}, y \rangle^n} \omega_y, \qquad \vol(L)^{-1}\mathbf{t}^* \mathscr{G}_0 = -\frac{(n-1)!}{\vol(L)} \frac{\omega_y}{\langle \mathbf{z}, y \rangle^n},
\end{equation}
and
\begin{equation}\label{eq:t_pullback_Gtilde}
    \mathbf{t}^* \widetilde{\mathscr{G}} = \frac{(-1)^{n-1} (n-1)!}{\vol(L)} \sum_{\mu \in L^*} \frac{e^{2\pi i \langle \mathbf{t}, \mu \rangle}}{\langle 2\pi i \mu - \mathbf{z}, y \rangle^n} \omega_y.
\end{equation}
\end{proposition}

\begin{proof}
By the uniform lower bound \eqref{eq:uniform_bound}, for $\Re(z_j) \leq -c < 0$ the coefficients $\langle 2\pi i \mu - \mathbf{z}, y \rangle^{-(a+1)}$ and all of their $y$-derivatives are smooth on $X$ and bounded in $\mu \in L^*$ by a polynomial in $\Vert \mu \Vert$ (in fact, uniformly bounded for order $0$ in $y$). Since a Fourier series $\sum_{\mu \in L^*} c_\mu e^{2\pi i \langle v, \mu \rangle}$ on the compact torus $T = V/L$ converges in the space of distributions if and only if $c_\mu$ has at most polynomial growth in $\Vert \mu \Vert$, the series $\mathscr{G}$ and $\widetilde{\mathscr{G}}$ converge unconditionally in $\widehat{\Omega}^{n-1}(X \times T)$. Since $\nabla$ is continuous on currents and $\delta_{X \times \{0\}} \vol_T = \sum_{\mu \in L^*} e^{2\pi i \langle v, \mu \rangle} \vol_T = \vol(L)^{-1} \sum_{\mu \in L^*} e^{2\pi i \langle v, \mu \rangle} dv_1 \wedge \cdots \wedge dv_n$, Lemma \ref{lem:fibre_integral_green}(1) yields $\nabla \mathscr{G} = \delta_{X \times \{0\}} \vol_T - \vol_T$ and $\nabla \widetilde{\mathscr{G}} = \delta_{X \times \{0\}} \vol_T$, so $\widetilde{\mathscr{G}}$ has residue $\delta_0$ along $X \times \{0\}$ by Corollary \ref{cor:explicit_residue} and Remark \ref{delta distribution}. Finally, since $\mathbf{t}^*(dv_j) = 0$ for all $j = 1, \dots, n$, we have $\mathbf{t}^*\vol_T = 0$ and all bidegree components $\mathscr{G}_\mu^a$ with $a < n-1$ vanish under $\mathbf{t}^*$, leaving only the $a = n-1$ component from Lemma \ref{lem:fibre_integral_green}(2).
\end{proof}

\subsection{Evaluation and comparison of polylogarithm and Shintani classes}\label{section evaluation and comparison}

Recall from Definition \ref{def shintani class} that for a Shintani decomposition $\Phi$ and a nonzero torsion point $\mathbf{t} \in \overset{\circ}{T}$ stabilized by $\Gamma$ (which we may assume lies outside the walls of the cones of $\Phi$ by Remark \ref{rmk cone avoids t}), the assignment $\langle \boldsymbol{\alpha} \rangle \mapsto \rho_N\bigl(\sgn(\boldsymbol{\alpha})\mathcal{G}_{\sigma_{\boldsymbol{\alpha}}}(\mathbf{t}, \mathbf{z})\bigr)$ defines the Shintani class
\[
    [\mathcal{G}] \in H^{n-1}\bigl(\Hom_{\mathbb{Z}}(C_\bullet(\Phi/\Gamma), \widetilde{R}_\Gamma)\bigr) \cong H^{n-1}(\Gamma, \widetilde{R}_\Gamma).
\]
Similarly, the assignment $\langle \boldsymbol{\alpha} \rangle \mapsto \rho_N\bigl(\vol(L)^{-1} f(\boldsymbol{\alpha})(\mathbf{z})\bigr)$ (or on oriented simplices $\langle \overline{\sigma}_{\boldsymbol{\alpha}} \rangle \mapsto \rho_N\bigl(\vol(L)^{-1}\sgn(\boldsymbol{\alpha}) f(\boldsymbol{\alpha})(\mathbf{z})\bigr)$) defines the \textit{polar cocycle class} $[\Psi_P] \in H^{n-1}(\Gamma, \widetilde{R}_\Gamma)$. By Lemma \ref{Lemma: quasi iso}, both classes are independent of the choice of Shintani decomposition $\Phi$.

At the torsion point $\mathbf{t} \in \overset{\circ}{T}$, the canonical multiplicative trivialization $\varrho_{\mathrm{can}}$ (which agrees with $\varrho_{\mathrm{cont}}|_{\mathbf{t}}$) identifies $\mathbf{t}^*\widetilde{\mathcal{L}\mathrm{og}} \cong \underline{\widetilde{R}}$. Our fixed orientation $J$ induces the evaluation map $\ev: H^{n-1}(\mathcal{S}^1, \widetilde{R}) \to \widetilde{R}_\Gamma$, given on a closed $(n-1)$-form $\eta$ on $\mathcal{S}^1 = V^1/\Gamma \cong X/\Gamma$ by
\[
    \ev([\eta]) := \rho_N\left(\int_{\mathcal{S}^1} \eta\right),
\]
which coincides under the de Rham--Betti comparison with the Kronecker pairing \eqref{eq: isom} against the fundamental class \eqref{eq: fundamental class}.

\begin{theorem}\label{main theorem}
Let $\mathbf{t} \in \overset{\circ}{T}$ be a nonzero torsion point stabilized by $\Gamma$. Let $\Eis(\mathbf{t}) := \mathbf{t}^*\pol = [\mathbf{t}^*\widetilde{\mathscr{G}}] \in H^{n-1}(\Gamma, \widetilde{R}_\Gamma)$ be the Eisenstein class obtained by pulling back $\pol$ along $\mathbf{t}$, and let $[\mathbf{t}^*\mathscr{G}] \in H^{n-1}(\Gamma, \widetilde{R}_\Gamma)$ be the cohomology class of the non-zero Fourier terms of the Green current. Then in $H^{n-1}(\Gamma, \widetilde{R}_\Gamma)$ we have
\begin{equation}\label{eq: equality of classes}
    [\mathbf{t}^*\mathscr{G}] = (-1)^{n-1}[\mathcal{G}] + [\Psi_P] \quad \text{and} \quad \Eis(\mathbf{t}) = [\mathbf{t}^*\widetilde{\mathscr{G}}] = (-1)^{n-1}[\mathcal{G}],
\end{equation}
and under the evaluation map \eqref{eq: isom} we have
\begin{equation}\label{eq: equality of evaluations}
    \ev([\mathbf{t}^*\mathscr{G}]) = \ev\bigl((-1)^{n-1}[\mathcal{G}] + [\Psi_P]\bigr) = (-1)^{n-1} \sum_{k \geq 0} \zeta(\mathbf{t}, L, -k) \frac{w^k}{(k!)^n} \in \mathbb{Q}\llbracket w \rrbracket.
\end{equation}
\end{theorem}

\begin{proof}
Under the radial diffeomorphism $\psi: \mathcal{S}^1 = V^1/\Gamma \xrightarrow{\sim} X/\Gamma$, the cohomology groups $H^{n-1}(X/\Gamma, \widetilde{R}_\Gamma) \cong H^{n-1}(\mathcal{S}^1, \widetilde{R}_\Gamma) \cong H^{n-1}(\Gamma, \widetilde{R}_\Gamma)$ are canonically identified with the simplicial cohomology $H^{n-1}\bigl(\Hom_{\mathbb{Z}}(C_\bullet(\Phi/\Gamma), \widetilde{R}_\Gamma)\bigr)$ via the triangulation $\Phi^X$ induced by $\Phi$.

Let $\sigma_{\boldsymbol{\alpha}} \in \Phi_n$ be an $n$-dimensional simplicial cone with generator $\boldsymbol{\alpha} = (\alpha_1, \dots, \alpha_n) \in L_{>0}^n$, defining the oriented simplex $\langle \sigma_{\boldsymbol{\alpha}} \rangle = \sgn(\boldsymbol{\alpha})\langle \boldsymbol{\alpha} \rangle$ and the compact spherical simplex $X_{\boldsymbol{\alpha}} = \sigma_{\boldsymbol{\alpha}} \cap X$. Because the Fourier series $\mathscr{G}(y,v)$ and $\widetilde{\mathscr{G}}(y,v)$ are smooth in $y \in X$ with coefficients uniformly bounded in $\mu \in L^*$ by \eqref{eq:uniform_bound}, integrating over the compact simplex $X_{\boldsymbol{\alpha}} \subset X$ commutes with the distributional sum over $\mu \in L^*$. Applying Hurwitz's formula (Proposition \ref{hurwitz}) term-by-term and evaluating at $v = \mathbf{t} \in V_{\sigma_{\boldsymbol{\alpha}}}/L$ (Remark \ref{rmk cone avoids t}) gives
\[
    \int_{X_{\boldsymbol{\alpha}}} \mathbf{t}^*\widetilde{\mathscr{G}} = \frac{(-1)^{n-1}}{\vol(L)} \sum_{\mu \in L^*} e^{2\pi i \langle \mathbf{t}, \mu \rangle} \sgn(\boldsymbol{\alpha}) f(\boldsymbol{\alpha})(2\pi i \mu - \mathbf{z}),
\]
\[
    \vol(L)^{-1}\int_{X_{\boldsymbol{\alpha}}} \mathbf{t}^*\mathscr{G}_0 = -\vol(L)^{-1}\sgn(\boldsymbol{\alpha}) f(\boldsymbol{\alpha})(\mathbf{z}),
\]
and
\[
    \int_{X_{\boldsymbol{\alpha}}} \mathbf{t}^*\mathscr{G} = \frac{(-1)^{n-1}}{\vol(L)} \sum_{\mu \in L^* \setminus \{0\}} e^{2\pi i \langle \mathbf{t}, \mu \rangle} \sgn(\boldsymbol{\alpha}) f(\boldsymbol{\alpha})(2\pi i \mu - \mathbf{z}).
\]
Applying the Poisson summation identity of Proposition \ref{prop: equality of distributions} (noting that $\sgn(\boldsymbol{\alpha})^2 = 1$ and that the $\mu = 0$ term satisfies $f(\boldsymbol{\alpha})(-\mathbf{z}) = (-1)^n f(\boldsymbol{\alpha})(\mathbf{z})$), we obtain on each oriented simplex $X_{\boldsymbol{\alpha}}$:
\[
    \int_{X_{\boldsymbol{\alpha}}} \mathbf{t}^*\widetilde{\mathscr{G}} = (-1)^{n-1} \mathcal{G}_{\sigma_{\boldsymbol{\alpha}}}(\mathbf{t}, \mathbf{z})
\]
and
\[
    \int_{X_{\boldsymbol{\alpha}}} \mathbf{t}^*\mathscr{G} = (-1)^{n-1} \mathcal{G}_{\sigma_{\boldsymbol{\alpha}}}(\mathbf{t}, \mathbf{z}) + \vol(L)^{-1} \sgn(\boldsymbol{\alpha}) f(\boldsymbol{\alpha})(\mathbf{z}).
\]
Applying $\rho_N$ gives the cochain-level equalities on $\langle \overline{\sigma}_{\boldsymbol{\alpha}} \rangle = \sgn(\boldsymbol{\alpha})\langle \boldsymbol{\alpha} \rangle$, proving \eqref{eq: equality of classes}.

Summing over a finite set of representatives $\sigma_{\boldsymbol{\alpha}} \in \Phi_n/\Gamma$ and applying $\rho_N$, the left-hand side of the second identity gives $\ev([\mathbf{t}^*\mathscr{G}]) = \rho_N\left(\int_{\mathcal{S}^1} \mathbf{t}^*\mathscr{G}\right)$. On the right-hand side, by Proposition \ref{prop: partial zeta decomp} and Lemma \ref{Lemma: gen Mellin} (noting that $\mathcal{G}_\sigma(\mathbf{t}, \mathbf{z})$ has $e^{+\langle \alpha, \mathbf{z} \rangle}$, so $t_j = -z_j$ cancels the sign $(-1)^{k_j}$ in Lemma \ref{Lemma: gen Mellin}), the diagonal coefficients in $w^k$ ($k \geq 0$) of $\rho_N\left(\sum_{\sigma \in \Phi_n/\Gamma} \mathcal{G}_\sigma(\mathbf{t}, \mathbf{z})\right)$ equal
\[
    \sum_{k \geq 0} \sum_{\sigma \in \Phi_n/\Gamma} \zeta_\sigma(\mathbf{t}, (-k, \dots, -k)) \frac{w^k}{(k!)^n} = \sum_{k \geq 0} \zeta(\mathbf{t}, L, -k) \frac{w^k}{(k!)^n},
\]
while the polar term $\sum_{\sigma \in \Phi_n/\Gamma} \vol(L)^{-1} \int_{X_\sigma} \mathbf{t}^*\mathscr{G}_0 = \vol(L)^{-1} \int_X \mathbf{t}^*\mathscr{G}_0 = -\vol(L)^{-1}\frac{1}{N(\mathbf{z})} = -\vol(L)^{-1}\frac{1}{w}$ (by Corollary \ref{corollary to hurwitz}) cancels the $k = -1$ polar term of $(-1)^{n-1}\rho_N([\mathcal{G}])$ to land in $F^{\mathrm{Gal}}\llbracket w \rrbracket$.

Finally, to see that the coefficients $\zeta(\mathbf{t}, L, -k)$ lie in $\mathbb{Q}$, note that both $\pol$ and $[\mathcal{G}]$ are defined over $\widetilde{R}_{F^{\mathrm{Gal}}}$, and any $\tau \in \Gal(F^{\mathrm{Gal}}/\mathbb{Q})$ acts on the real embeddings $J_1, \dots, J_n$ by permuting the coordinate indices $\{1, \dots, n\}$. Since the diagonal monomial $w^k = (z_1 \cdots z_n)^k$ is symmetric under all permutations of $\{1, \dots, n\}$, the diagonal coefficients of $\rho_N([\mathcal{G}])$ are invariant under $\Gal(F^{\mathrm{Gal}}/\mathbb{Q})$ and hence belong to $(F^{\mathrm{Gal}})^{\Gal(F^{\mathrm{Gal}}/\mathbb{Q})} = \mathbb{Q}$ (compare also \cite[(5.9)]{BKL18}).
\end{proof}

\begin{corollary}[Klingen--Siegel]\label{cor: klingen siegel}
The values of partial zeta functions of totally real number fields at nonpositive integers are rational: for any nonzero integral ideals $\mathfrak{f}, \mathfrak{b} \subset \mathcal{O}_F$, we have
\[
    \zeta(\mathfrak{b}, \mathfrak{f}, -k) \in \mathbb{Q}
\]
for all $k \geq 0$ if $\mathfrak{f} \neq \mathcal{O}_F$, and for all $k \geq 1$ if $\mathfrak{f} = \mathcal{O}_F$.
\end{corollary}

\begin{proof}
If $\mathfrak{f} \neq \mathcal{O}_F$, then $1$ is a nonzero torsion point in $(F \otimes_{\mathbb{Q}} \mathbb{R})/\mathfrak{f}\mathfrak{b}^{-1}$, so the claim follows from Theorem \ref{main theorem} and \eqref{lattice zeta and partial}. If $\mathfrak{f} = \mathcal{O}_F$ and $k \geq 1$, choosing a prime ideal $\mathfrak{p} \subset \mathcal{O}_F$ and removing the Euler factor $(1 - N(\mathfrak{p})^k)$ via
\[
    \sum_{\substack{\mathfrak{b}' \bmod \mathfrak{p} \\ \mathfrak{b}' \equiv \mathfrak{b} \bmod \mathcal{O}_F}} \zeta(\mathfrak{b}', \mathfrak{p}, s) = \bigl(1 - N(\mathfrak{p})^{-s}\bigr) \zeta(\mathfrak{b}, \mathcal{O}_F, s)
\]
reduces the claim to the case $\mathfrak{f} = \mathfrak{p} \neq \mathcal{O}_F$.
\end{proof}

\section{Comparison with Previous Literature}\label{sec:literature_comparison}

\subsection{Comparison with Beilinson, Kings, and Levin \cite{BKL18}}\label{subsection:BKL comparison}

\paragraph{Variants of the topological polylogarithm and the $\mu = 0$ term.}
In \cite{BKL18}, the logarithm sheaf $\mathcal{L}\mathrm{og}$ is defined over the completed group ring $R_A = \varprojlim_k A[L]/J^k$. Because $H^n(T, \mathcal{L}\mathrm{og} \otimes \lambda_L) \cong A$ does not vanish over $R_A$, the localization sequence on $T \setminus D$ reads
\[
    0 \longrightarrow H^{n-1}(T \setminus D, \mathcal{L}\mathrm{og} \otimes \lambda_L) \xrightarrow{\res} \bigoplus_{d \in D} \mathcal{L}\mathrm{og}_d \xrightarrow{\sigma_D} A \longrightarrow 0.
\]
Consequently, \cite{BKL18} introduces two variants of the topological polylogarithm over $R_A$:
\begin{enumerate}
    \item First (\cite[Definition 3.30]{BKL18}), for a finite $\Gamma$-stable subset $D \subset T^{(A)}$ and a degree-zero divisor $\alpha \in (A[D]^0)^\Gamma = \ker(\Sigma: A[D] \to A)^\Gamma$, the class $\pol_\alpha \in H^{n-1}(T \setminus D, \Gamma; \mathcal{L}\mathrm{og} \otimes \lambda_L)$ with $\res(\pol_\alpha) = \alpha$.
    \item Second (\cite[Definition 3.40]{BKL18}), when $A$ is a $\mathbb{Q}$-algebra and $D = \{0\}$, twisting by $L_A^* = \Hom_A(L_A, A)$ produces a class $\pol_{\mathrm{BKL}} \in H^{n-1}(T \setminus \{0\}, \Gamma; \underline{L}_A^* \otimes_A \mathcal{L}\mathrm{og} \otimes \lambda_L)$ whose residue at $0$ is the canonical element $\varpi \in (L_A^* \otimes_A L_A)^\Gamma \subset (L_A^* \otimes_A I)^\Gamma$ corresponding to $\id_{L_A}$.
\end{enumerate}
In our setting, passing from $R_{\mathbb{K}} \cong \mathbb{K}\llbracket \mathbf{z} \rrbracket$ to the cone-adapted Laurent field $\widetilde{R}_{\mathbb{K}} = \mathbb{K}_{\leq_{\mathrm{lex}}}(\!(\mathbf{z})\!)$ inverts the augmentation ideal $I$, forcing $H^n(T, \widetilde{\mathcal{L}\mathrm{og}}) = 0$ (Proposition \ref{prop: vanishing coh}). This makes the residue map at $D = \{0\}$ an isomorphism (Corollary \ref{cor: localization coh}) and allows us to define $\pol \in H^{n-1}(T \setminus \{0\}, \Gamma; \widetilde{\mathcal{L}\mathrm{og}})$ directly by $\res(\pol) = \delta_0$ without twisting by a degree-zero divisor $\alpha$ or by $L_A^*$. However, we note that the twist by a degree-zero divisor seems essential to get integrality results.

At the level of Green currents, our normalized sum over non-zero dual lattice points $\mathscr{G} = \vol(L)^{-1}\sum_{\mu \in L^* \setminus \{0\}} \mathscr{G}_\mu$ in Proposition \ref{prop green current} satisfies $\nabla\mathscr{G} = \delta_{X \times \{0\}}\vol_T - \vol_T$, which is the counterpart on $X \times T$ of the Green current of \cite[Definition 4.3 and Corollary 4.14]{BKL18}. In \cite[Theorem 4.7]{BKL18}, the $-\vol_T$ term is cancelled by adding a closed form $\eta$ satisfying $-\kappa\eta = \varpi\vol_T$ to obtain $\varpi\mathscr{G}_{\mathrm{BKL}} + \eta$. Over $\widetilde{R}_{\mathbb{C}}$, the normalized $\mu = 0$ Fourier term $\vol(L)^{-1}\mathscr{G}_0$ is well-defined and satisfies $\nabla(\vol(L)^{-1}\mathscr{G}_0) = \vol_T$, so $\widetilde{\mathscr{G}} = \mathscr{G} + \vol(L)^{-1}\mathscr{G}_0$ satisfies $\nabla\widetilde{\mathscr{G}} = \delta_{X \times \{0\}}\vol_T$. Upon pulling back by $\mathbf{t}$ and integrating over a Shintani simplex $X_{\boldsymbol{\alpha}}$, $\vol(L)^{-1}\mathbf{t}^*\mathscr{G}_0$ evaluates to $-\vol(L)^{-1}\sgn(\boldsymbol{\alpha})f(\boldsymbol{\alpha})(\mathbf{z})$, explaining why $[\mathbf{t}^*\mathscr{G}]$ differs from $(-1)^{n-1}[\mathcal{G}]$ by the polar cocycle $[\Psi_P]$ while $[\mathbf{t}^*\widetilde{\mathscr{G}}] = (-1)^{n-1}[\mathcal{G}]$.

\paragraph{Elimination of the space $\mathscr{P}$ and comparison with \cite[Theorem 5.6]{BKL18}.}
In \cite[\S 4.3]{BKL18}, working over $R_{\mathbb{C}} = \mathbb{C}\llbracket \mathbf{z} \rrbracket$ and following Nori (\cite{Nor95}), the authors introduce the contractible space $\mathscr{P} \subset \Hom(V, V^*)$ of symmetric positive-definite bilinear forms $B$ on $V$ and the $\mu$-dependent maps $v_\mu: \mathscr{P} \to V_{\mu > 0}$, $B \mapsto B^{-1}(\mu)$, because $2\pi i \langle \mu, v \rangle - \langle \mathbf{z}, v \rangle$ is invertible in $\mathbb{C}\llbracket \mathbf{z} \rrbracket$ if and only if $\langle \mu, v \rangle \neq 0$ (which holds for $v = B^{-1}(\mu)$ since $B^{-1}(\mu, \mu) > 0$). To evaluate $\mathbf{t}^*\mathscr{G}_{\mathrm{BKL}}$ in \cite[Theorem 5.6]{BKL18}, they embed $V^1 \hookrightarrow \mathscr{P}$ via $q \mapsto B_q$, where $B_q(v,w) := \sum_{j=1}^n q_j^{-1} v_j w_j$. Under this embedding, $v_\mu(B_q) = (q_1\mu_1, \dots, q_n\mu_n)$ multiplies the $j$-th coordinate $q_j$ by $\mu_j$, so the pullback of \cite[Lemma 4.12]{BKL18} to $V^1$ is given in \cite[(5.7)]{BKL18} by
\[
    \mathbf{t}^*\mathscr{G}_{\mathrm{BKL}} = (-1)^{n-1} (n-1)! \sum_{\mu \in L^* \setminus \{0\}} e^{2\pi i \langle \mathbf{t}, \mu \rangle} \frac{N(\mu)\, \omega_{V^1}}{\left(\sum_{j=1}^n (2\pi i \mu_j^2 - \mu_j z_j)q_j\right)^n}.
\]
Because of the extra factors of $\mu_j$ in $v_\mu(B_q)$, $\mathbf{t}^*\mathscr{G}_{\mathrm{BKL}}$ does not integrate over an individual Shintani simplex $X_{\boldsymbol{\alpha}}$ to Sczech's rational function $f(\boldsymbol{\alpha})(2\pi i\mu - \mathbf{z})$; instead, in \cite[p.~39]{BKL18}, those factors of $\mu_j$ only cancel after unfolding the integral over $V^1/\Gamma$ to the entire positive orthant $V_{>0} = \mathbb{R}_{>0}^n$ and invoking the functional equation for partial zeta functions (\cite[Proposition 5.1 and Corollary 5.2]{BKL18}).

Over $\widetilde{R}_{\mathbb{C}}$, by contrast, $\langle 2\pi i \mu - \mathbf{z}, y \rangle$ is already invertible for all $y \in X$ and all $\mu \in L^*$ (Section \ref{sec:direct_green}). This allows us to bypass $\mathscr{P}$ and the $\mu$-dependent maps $v_\mu$ altogether, constructing $\mathscr{G}_\mu$ directly on $X \times T$ via the single $\mu$-independent map $\Phi(y,v,u) = v - uy$. The resulting specialization $\mathbf{t}^*\mathscr{G}$ in Proposition \ref{prop green current} has denominator $\langle 2\pi i \mu - \mathbf{z}, y \rangle^n$ rather than $\langle 2\pi i \mu^2 - \mu\mathbf{z}, q \rangle^n$, so its integral over each Shintani simplex $X_{\boldsymbol{\alpha}}$ yields $f(\boldsymbol{\alpha})(2\pi i \mu - \mathbf{z})$ directly via Hurwitz's local Mellin formula (Proposition \ref{hurwitz}) and recovers the Shintani generating functions via Poisson summation on each cone (Proposition \ref{prop: equality of distributions}).

\subsection{Comparison with Charollois, Dasgupta, and Greenberg \cite{CDG15}}\label{subsection:CDG comparison}

\paragraph{Dictionary}

We extend our previous diagram \eqref{identification} into the following identifications:
\begin{equation}\label{identification_CDG}
\begin{tikzcd}[row sep=tiny]
    \mathbb{Q}^n \otimes_{\mathbb{Q}} \mathbb{R} \arrow[rr, "\cdot \mathbf{w}", "\cong"'] && F \otimes_{\mathbb{Q}} \mathbb{R} \arrow[rr, "J", "\cong"'] && \mathbb{R}^n =: V \\
    \mathbb{Z}^n \otimes 1 \arrow[rr, maps to] && \mathfrak{f}\mathfrak{b}^{-1} \otimes 1 \arrow[rr, maps to] && L \\
    v + \mathbb{Z}^n \arrow[rr, maps to] && 1 + \mathfrak{f}\mathfrak{b}^{-1} \arrow[rr, maps to] && 1 + L \\
    Q \arrow[rrrr, maps to] &&&& (0, \dots, 0, 1),
\end{tikzcd}
\end{equation}
where $v := (\Tr(w_1^*), \dots, \Tr(w_n^*))$ and $Q := (J_n(w_1^*), \dots, J_n(w_n^*))$ for the dual basis $(w_1^*, \dots, w_n^*)$ of $F$ with respect to the trace pairing $\Tr(x y)$, which corresponds under $J$ to the standard Euclidean inner product $\langle \cdot, \cdot \rangle$ on $V = \mathbb{R}^n$.

In \cite{CDG15}, the authors work on $\mathbb{Q}^n \otimes_{\mathbb{Q}} \mathbb{R}$ corresponding to the left-hand side of our identification diagram \eqref{identification_CDG}, with the matrix $M$ in their notation corresponding to $J(\mathbf{w})$ (normalized by $N(\mathfrak{b})$). 

When considering Shintani decompositions, they work with open cones, which requires considering cones of every dimension to account for the removed faces of the closed cones in top-dimension which we consider. Furthermore, under \eqref{identification_CDG}, our upper closure $\widehat{\sigma}_{\boldsymbol{\alpha}}$ (which removes the lower faces with respect to the last coordinate $(0, \dots, 0, 1)$) corresponds precisely to their $Q$-perturbation $c(\boldsymbol{\alpha})_Q$ for $Q = (J_n(w_1^*), \dots, J_n(w_n^*))$ (\cite[\S 2.1]{CDG15}). The notion of a $Q$-perturbation depending on a variable $Q$, however, gives them extra flexibility which is required for their integrality results.

\paragraph{Group actions.}
The multiplicative action of $F^\times$ on $F$ induces a representation $\rho_{\mathbf{w}}: F^\times \to \GL_n(\mathbb{Q})$ defined by $u \cdot \mathbf{w} = \mathbf{w} \cdot \rho_{\mathbf{w}}(u)$ for $u \in F^\times$, making the identifications in \eqref{identification_CDG} equivariant:
\begin{equation}\label{action identification}
\begin{tikzcd}
    \mathbb{Q}^n \otimes_{\mathbb{Q}} \mathbb{R} \arrow[loop above, "{\rho_{\mathbf{w}}(F^\times)}"] \arrow[rr, "\cdot \mathbf{w}", "\cong"'] && F \otimes_{\mathbb{Q}} \mathbb{R} \arrow[loop above, "{F^\times}"] \arrow[rr, "J", "\cong"'] && \mathbb{R}^n \arrow[loop above, "{\diag(J(F^\times))}"],
\end{tikzcd}
\end{equation}
where $u \in F^\times$ acts on $\mathbb{Q}^n$ on the right via multiplication by $\rho_{\mathbf{w}}(u)^\top$. Since $\det(\rho_{\mathbf{w}}(u)) = N(u) = 1$ for $u \in \Gamma \subset \mathcal{O}_F^{\times,+}$, the action of $\Gamma$ preserves $L$, $V_{>0}$, and $V^1$:
\begin{equation}\label{integral action}
\begin{tikzcd}
    \mathbb{Z}^n \arrow[loop above, "{\rho_{\mathbf{w}}(\Gamma)}"] \arrow[rr, "\cdot \mathbf{w}", "\cong"'] && \mathfrak{f}\mathfrak{b}^{-1} \arrow[loop above, "{\Gamma}"] \arrow[rr, "J", "\cong"'] && L \arrow[loop above, "{\diag(J(\Gamma))}"].
\end{tikzcd}
\end{equation}

In \cite[\S 2.2 and \S 3.1]{CDG15}, they consider the module $\mathcal{K}$ of functions on $V$ generated by characteristic functions of cones. Under the equivariant identification \eqref{action identification}, their $(n-1)$-cocycle $\phi_U$ for $\Gamma$ and their $\GL_n(\mathbb{Q})$-cocycle $\Phi_{\mathrm{Sh}}$ are related by $J(\mathbf{w})^*[\phi_U] = \rho_{\mathbf{w}}^*[\Phi_{\mathrm{Sh}}]$. Applying the Solomon--Hu pairing (\cite[Proposition 3.1]{CDG15}) to $\Phi_{\mathrm{Sh}}$ produces their Shintani cocycle $[\Psi_{\mathrm{Sh}}]$, which evaluates on each cone $\sigma_{\boldsymbol{\alpha}}$ to our cone generating series $\mathcal{G}_{\sigma_{\boldsymbol{\alpha}}}(\mathbf{t}, \mathbf{z})$ of Definition \ref{def: generating}.

\paragraph{Ring of coefficients and the Shintani operator $\Theta$.}
To extract zeta values from the meromorphic series $\mathcal{G}_{\sigma_{\boldsymbol{\alpha}}}(\mathbf{t}, \mathbf{z})$, \cite[Definition 3.3]{CDG15} introduces the subalgebra $\mathbb{K}(\!(\mathbf{z})\!)^{hd} \subset \Frac(\mathbb{K}\llbracket \mathbf{z} \rrbracket)$ consisting of fractions $G = \frac{g}{p}$ where $g \in \mathbb{K}\llbracket \mathbf{z} \rrbracket$ and $p(\mathbf{z})$ is a homogeneous polynomial of degree $d$ in which the pure power $z_i^d$ appears with nonzero coefficient for every $i = 1, \dots, n$ (such as $p_A(\mathbf{z}) = N(A\mathbf{z}) = \prod_{i=1}^n L_i(\mathbf{z})$). For each $i = 1, \dots, n$, the substitution
\[
    \theta_i(\mathbf{z}) := (z_i z_1, \dots, z_i z_{i-1}, z_i, z_i z_{i+1}, \dots, z_i z_n)
\]
maps $p(\theta_i(\mathbf{z}))$ into $z_i^d \mathbb{K}\llbracket z_1, \dots, \widehat{z}_i, \dots, z_n \rrbracket^\times$, so $\theta_i(G)(\mathbf{z}) := G(\theta_i(\mathbf{z})) \in z_i^{-d}\mathbb{K}\llbracket \mathbf{z} \rrbracket$. Composing $\theta = (\theta_1, \dots, \theta_n)$ with the $i$-th shifted diagonal projections $\rho_{N,i}: z_i^{-d}\mathbb{K}\llbracket \mathbf{z} \rrbracket \to \mathbb{K}(\!(w)\!)$ sending $(z_1 \cdots z_{i-1} z_i^n z_{i+1} \cdots z_n)^k = \theta_i(z_1 \cdots z_n)^k \mapsto w^k$ (and all other monomials to $0$) and averaging over $i = 1, \dots, n$ defines the operator
\[
    \Theta: \mathbb{K}(\!(\mathbf{z})\!)^{hd} \xrightarrow{\theta} \prod_{i=1}^n z_i^{-d}\mathbb{K}\llbracket \mathbf{z} \rrbracket \xrightarrow{(\rho_{N,i})_{i \leq n}} \prod_{i=1}^n \mathbb{K}(\!(w)\!) \xrightarrow{\Avg} \mathbb{K}(\!(w)\!),
\]
which equals $\sum_{k \in \mathbb{Z}} \frac{\Delta^{(k)}}{(k!)^n} w^k$ for the Shintani operator $\Delta$ of \cite[Definition 3.5]{CDG15}.

\begin{lemma}\label{lem:CDG_Theta_commutes}
Fix $J$ and $\mathbf{w}$ as in \eqref{identification_CDG}. Then $\mathbb{K}(\!(\mathbf{z})\!)^{hd}$ embeds naturally into $\widetilde{R}_{\mathbb{K}} = \mathbb{K}_{\leq_{\mathrm{lex}}}(\!(\mathbf{z})\!)$, and the following diagram commutes:
\begin{equation}\label{coefficient diagram}
\begin{tikzcd}
    R_{\mathbb{K}} \arrow[rr, hook] \arrow[dd, "{\rho_\Gamma = \rho_N}"'] && {\mathbb{K}(\!(\mathbf{z})\!)^{hd}} \arrow[rr, hook] \arrow[ddrr, "\Theta"'] && {\mathbb{K}_{\leq_{\mathrm{lex}}}(\!(\mathbf{z})\!)} \arrow[dd, "{\rho_N}", two heads] \\
    \\
    {\mathbb{K}\llbracket w \rrbracket} \arrow[rrrr, hook] &&&& {\mathbb{K}(\!(w)\!)}.
\end{tikzcd}
\end{equation}
\end{lemma}

\begin{proof}
For $G(\mathbf{z}) = \sum_{\mathbf{e} \in \mathbb{Z}^n} a_{\mathbf{e}} \mathbf{z}^{\mathbf{e}} \in \mathbb{K}_{\leq_{\mathrm{lex}}}(\!(\mathbf{z})\!)$, the substitution $\theta_i$ acts on the exponent lattice $\mathbb{Z}^n$ via the automorphism $\theta_i(e_1, \dots, e_n) = (e_1, \dots, e_i + \sum_{j \neq i} e_j, \dots, e_n)$, so
\[
    G(\theta_i(\mathbf{z})) = \sum_{\mathbf{e} \in \mathbb{Z}^n} a_{(e_1, \dots, e_i - \sum_{j \neq i} e_j, \dots, e_n)} \mathbf{z}^{\mathbf{e}}.
\]
Applying $\rho_{N,i}$ after $\theta_i$ extracts the coefficient of $(z_1 \cdots z_{i-1} z_i^n z_{i+1} \cdots z_n)^k$ in $G(\theta_i(\mathbf{z}))$ (since $z_i^{-d}$ shifts the $i$-th exponent), which corresponds to the exponent $\mathbf{e} = (k, \dots, nk, \dots, k)$ and thus selects the coefficient $a_{(k, \dots, nk - (n-1)k, \dots, k)} = a_{(k, \dots, k)}$. Hence every $\theta_i$ selects the same diagonal coefficient $a_{(k, \dots, k)}$, and their average $\Theta(G)$ coincides with $\rho_N(G)$.
\end{proof}

Consequently, Theorem 4.1 of \cite{CDG15} (which compares the Sczech cocycle with their Shintani cocycle $[\Psi_{\mathrm{Sh}}]$ up to the polar cocycle $[\Psi_P]$) corresponds under Lemma \ref{lem:CDG_Theta_commutes} to our Theorem \ref{main theorem}.

\paragraph{Distribution property of $\mathcal{F}$ and trace compatibility.}
In \cite[\S 3.5]{CDG15}, the cocycle $[\Psi_{\mathrm{Sh}}]$ takes values in a space $\mathcal{F}$ of functions on $\GL_n(\mathbb{Q})$ with values in $\mathbb{R}(\!(\mathbf{z})\!)^{hd}$ satisfying a distribution relation. In our geometric framework, $\mathcal{F}$ corresponds to sections of $(V \times \widetilde{R})/(L \rtimes \Gamma)$, and the $\GL_n(\mathbb{Q})$-action and distribution relation on $\mathcal{F}$ correspond precisely to the trace compatibility of the topological polylogarithm under isogenies (\cite[\S 3.11]{BKL18}).

Explicitly, for fixed $Q, M = J(\mathbf{w}), v$ as in \eqref{identification_CDG}, the distribution property for $f \in \mathcal{F}$ states that for any $\lambda \in \mathbb{Z} \setminus \{0\}$,
\[
    f(M, v)(\mathbf{z}) = \sgn(\lambda)^n \sum_{l \in (\mathbb{Z}/\lambda\mathbb{Z})^n} f\bigl(M, \lambda^{-1}l + \lambda^{-1}v\bigr)(\mathbf{z}\lambda M^\top).
\]
More generally, any matrix $A \in \GL_n(\mathbb{Q}) \cap M_n(\mathbb{Z})$ defines an isogeny of lattices $A: L \to L$ via \eqref{identification_CDG} in the sense of \cite[Definition 3.4]{BKL18}, and the $\GL_n(\mathbb{Q})$-action
\[
    (A \cdot f)(M, v)(\mathbf{z}) = \sum_{r \in L/AL} \sgn(\det A)\, f(A^\top M, A^{-1}r + A^{-1}v)(\mathbf{z})
\]
is the explicit realization of the trace map $\Tr_{\varphi}$ of \cite[Proposition 3.45]{BKL18}. When $A = \rho_{\mathbf{w}}(\gamma)$ for $\gamma \in \Gamma \hookrightarrow \SL_n(\mathbb{Z})$, the cokernel $L/AL$ is trivial and $\det A = 1$, recovering the semidirect product action of $L \rtimes \Gamma$ of \cite[\S 3.7]{BKL18}.


\end{document}